\documentclass{amsart}
\usepackage[utf8]{inputenc}
\usepackage{amssymb}
\usepackage[english]{babel}
\usepackage{mathrsfs}
\usepackage{euscript}
\usepackage{graphicx}
\usepackage{orcidlink}

\usepackage{amsaddr}
\usepackage{amssymb}
\usepackage{amsthm}
\usepackage{amsfonts}
\usepackage{amsbsy}
\usepackage{enumitem}
\usepackage{xcolor}
\renewcommand{\le}{\leqslant}
\renewcommand{\ge}{\geqslant}

\newcommand{\tr}{\mathrm{tr}}
\renewcommand{\mod}{\,\mathrm{mod}\,}

\newtheorem{theorem}{Theorem}[section]

\newtheorem{lemma}[theorem]{Lemma}
\newtheorem{remark}[theorem]{Remark}
\newtheorem{observation}[theorem]{Observation}
\newtheorem{corollary}[theorem]{Corollary}
\theoremstyle{definition}
\newtheorem{example}[theorem]{Example}

\newcommand{\tocorrect}[1]{}
\setlist[enumerate,1]{label=(\roman*)}
\numberwithin{equation}{section}

\begin{document}
	\title{$m$-isometric weighted shifts with operator weights}
	\author{Michał Buchała}
	\email{mbuchala@agh.edu.pl}	
	\address{AGH University of Krakow, Faculty of Applied Mathematics, al. A. Mickiewicza 30, 30-059 Krakow}
	\begin{abstract}
		The aim of this paper is to study $ m $-isometric weighted shifts with operator weights (both unilateral and bilateral). We obtain a characterization of such shifts by polynomials with operator coefficients. The procedure of construction of all $ m $-isometric weighted shifts with positive weights is presented. We answer the question when the weights of $ m $-isometric shift are commuting. The completion problem for $ m $-isometric weighted shifts with operator weights is solved. We characterize $ m $-isometric bilateral shifts the adjoints of which are also $ m $-isometric.
		
	\end{abstract}
   	\maketitle
	\section{Introduction}
\label{SecIntroduction}
Classical weighted shifts (both unilateral and bilateral) form an important class of operators on Hilbert space, which is a rich source of examples in operator theory. They have been studied extensively for many years (see \cite{ShieldsWeightedShiftOperators1974} for comprehensive survey in this topic). In this paper we deal with a certain generalization of such operators, namely weighted shifts with operator weights. We consider unilateral weighted shifts:
\begin{equation*}
    \ell^{2}(\mathbb{N},H)\ni(h_{0},h_{1},\ldots)\longmapsto (0,S_{1}h_{0},S_{2}h_{1},\ldots)\in \ell^{2}(\mathbb{N},H)
\end{equation*}
and bilateral weighted shifts:
\begin{equation*}
    \ell^{2}(\mathbb{Z},H)\ni(\ldots,h_{-1},\boxed{h_{0}},h_{1},\ldots)\longmapsto (\ldots,S_{-1}h_{-2},\boxed{S_{0}h_{-1}},S_{1}h_{0},\ldots)\in \ell^{2}(\mathbb{Z},H),
\end{equation*}
where $ \boxed{\,\cdot \,} $ denotes the zeroth term of the two-sided sequence.
In \cite{lambertUnitaryEquivalence1971} Lambert investigated unitary equivalence and reducibility of unilateral weighted shifts with operator weights (see \cite{pilidiReducibilityOfTwoSidedWeightedShift1978} and \cite{buchalaUnitarilyEquivalentBilateralWeightedShifts2024} for research on similar subject in case of bilateral shifts). The same author described the weakly-closed algebra generated by such an operator (see \cite{lambertAlgebraGeneratedByInvertiblyShift1972}). \par
$ m $-isometric operators were introduced by Agler in \cite{aglerDisconjugacyTheoremToeplitz1990} and investigated further in \cite{aglerMisometricTransformationsHilbert1995} (and its subsequent parts). From that time on $ m $-isometries are subject of long-term research in operator theory (see \cite{bermudezProductsMisometries2013}, \cite{jablonskiMisometricOperatorsTheir2020}, \cite{koMisometricToeplitzOperators2018}, \cite{suciuOperatorsExpansiveMisometric2022},\cite{badeaCauchyDual2isometric2019}, \cite{bermudezLocalSpectral2024}). In \cite{jablosnkiHyperexpansiveOperatorValuedUWS} Jabłoński characterized 2-isometric unilateral weighted shifts with operator weights. In \cite{anandSolutionCauchyDual2019}, on the occasion of solving the Cauchy dual subnormality problem for 2-isometries, Anand et.al. showed that a non-unitary 2-isometric operator satisfying so-called kernel condition is (up to a unitary equivalence) the orthogonal sum of unitary operator and unilateral weighted shift with positive and invertible weights (see \cite{kosmiderWoldTypeDecomposition2021} for certain generalization of this fact in case of general $ m $-isometries).\par
$ m $-isometric weighted shifts (with scalar weights) were characterized first by Berm\'{u}dez et.al. in \cite{bermudezWeightedShiftsWhich2010}. Some time later, in \cite{abdullahStructureMisometricWeighted2016} there appeared a more convenient characterization of $ m $-isometric shifts by real polynomials of degree at most $ m-1 $. Our goal is to generalize the results of \cite{abdullahStructureMisometricWeighted2016}. The paper is organized as follows. In Section \ref{SecPreliminaries} we set up notation and terminology and review some of the standard facts on weighted shifts with operator weights. Section \ref{SecMIsometricUWS} is devoted to study $ m $-isometric unilateral weighted shifts. Theorem \ref{ThmCharOfMIsometricUWSByOperatorPolynomial} characterizes $ m $-isometric unilateral weighted shifts with operator weights by polynomials of degree at most $ m-1 $ with operator coefficients; it may be viewed as a certain generalization of \cite[Theorem 1]{abdullahStructureMisometricWeighted2016}. In Theorem \ref{ThmConstructingMIsometricUWS} we give the method of constructing all $ m $-isometric weighted shifts with positive and invertible weights. Theorem \ref{ThmWhenWeightsCommuteUWS} answers the question when the weights of $ m $-isometric unilateral weighted shifts commute; it turns out that they commute if and only if the coefficients of polynomial given by Theorem \ref{ThmCharOfMIsometricUWSByOperatorPolynomial} commute. We also solve the completion problem for $ m $-isometric unilateral weighted shifts with operator weights (see Theorem \ref{ThmCompletingWeightsToMIsometricUWS}). Several examples are also provided. The aim of Section \ref{SecMIsometricBWS} is to provide counterparts of results presented in Section \ref{SecMIsometricUWS} for bilateral weighted shifts. Beyond that it is shown that the adjoint of $ m $-isometric bilateral weighted shift with operator weights is also $ m $-isometric if and only if the polynomial given by Theorem \ref{ThmConstructingMIsometricBWS} is invertible (as an element of the algebra of polynomials with operator coefficients). This allows us to answer the question in \cite[Problem 4.9]{bermudezLocalSpectral2024} in negative.

	\section{Preliminaries}
\label{SecPreliminaries}

By $ \mathbb{Z} $ we denote the set of integers and by $ \mathbb{N} $ the set of non-negative integers. For $ p\in \mathbb{Z} $ we define
\begin{equation*}
    \mathbb{Z}_{p} = \left\{ n\in \mathbb{Z}\colon n\ge p \right\}.
\end{equation*}
We denote by $ \mathbb{R} $ and $ \mathbb{C} $ the fields of real and complex numbers, respectively. For simplicity of notation we write $ (a_{i})_{i\in I} \subset A $, if $ (a_{i})_{i\in I} $ is the sequence of elements of a set $ A $. If $ A $ is $ n\times n $ matrix, then $ \tr\, M $ stands for the trace of $ M $. \par
If $ R $ is an algebra over a field $ \mathbb{K} \in \left\{ \mathbb{R}, \mathbb{C} \right\} $, then $ R[z] $ stands for the algebra of all polynomials with coefficients in $ R $ with standard addition and multiplication. For $ n\in \mathbb{N} $, $ R_{n}[z] $ denotes the vector space over $ \mathbb{K} $ of all polynomials with coefficients in $ R $ of degree less than or equal to $ n $.

Throughout this paper all Hilbert spaces are assumed to be complex. Let $ H $ be a Hilbert space. By $ \mathbf{B}(H) $ we denote the $ C^{\ast} $-algebra of all linear and bounded operators on $ H $. For $ T\in \mathbf{B}(H) $, we will write $ \mathcal{N}(T) $ for the kernel and $ \mathcal{R}(T) $ for the range of $ T $; by $ \sigma(T) $ and $ r(T) $ we denote the spectrum and the spectral radius of $ T $, respectively. For $ T\in \mathbf{B}(H) $ we have the following so-called range-kernel decomposition:
\begin{equation}
    \label{FormRangeKernelDecom}
    H = \mathcal{N}(T) \oplus \overline{\mathcal{R}(T^{\ast})}.
\end{equation}
An operator $ U\in \mathbf{B}(\mathcal{H}) $ is called a partial isometry if $ U|_{\mathcal{N}(U)^{\perp}} $ is isometric. For every $ T\in \mathbf{B}(\mathcal{H}) $ there exists the unique partial isometry $ U\in \mathbf{B}(\mathcal{H}) $ such that $ \mathcal{N}(T) = \mathcal{N}(U) $ and $ T = U\lvert T\rvert $, where $ \lvert T\rvert = (T^{\ast}T)^{1/2} $ (see \cite[Theorem 7.20]{weidmannLinearOperatorsInHilbertSpace}); the decomposition $ T = U\lvert T\rvert $ is called the polar decomposition of $ T $. The operators $ T,S\in \mathbf{B}(H) $ are unitarily equivalent if there exists a unitary operator $ U\in \mathbf{B}(H) $ such that $ US = TU $. Note that if $ T $ and $ S $ are unitarily equivalent, then so are $ T^{\ast} $ and $ S^{\ast} $ (by the same unitary operator). The operators $ T,S\in \mathbf{B}(H) $ are doubly commuting if $ TS = ST $ and $ T^{\ast}S = ST^{\ast} $.
An operator $ T\in \mathbf{B}(H) $ is called $ m $-isometry ($ m\in \mathbb{Z}_{1} $) if
\begin{equation}
    \label{FormDefMIsometricOperator}
    \beta_{m}(T):= \sum_{k=0}^{m}\binom{m}{k}(-1)^{m-k}T^{\ast k}T^{k} = 0;
\end{equation}
$ T $ is called positive if $ \langle Tf,f\rangle > 0 $ for every $ f\in H $, $ f\not=0 $. 
For $ N\in \left\{ \mathbb{N}, \mathbb{Z} \right\} $ we define
\begin{equation*}
    \ell^{2}(N,H) = \{(h_{i})_{i\in N} \subset H\colon \sum_{i\in N}\lVert h_{i}\rVert^{2} < \infty \};
\end{equation*}
this is a Hilbert space with the inner product given by the formula
\begin{equation*}
    \langle h,h'\rangle = \sum_{i\in N} \langle h_{i},h_{i}'\rangle, \quad h = (h_{i})_{i\in N},h' = (h_{i}')_{i\in N}\in \ell^{2}(N,H).
\end{equation*}
Every operator $ T\in \mathbf{B}(\ell^{2}(N,H)) $ has a matrix representation $ [T_{i,j}]_{i,j\in N} $, where $ T_{i,j}\in \mathbf{B}(H) $, satisfying the following formula (see \cite[Chapter 8]{halmosHilbertSpaceProblemBook}):
\begin{equation*}
    T(x_{i})_{i\in N} = \left( \sum_{j\in N} T_{i,j} x_{j} \right)_{i\in N}, \quad (x_{i})_{i\in N} \in \ell^{2}(N,H).
\end{equation*}
Observe that if $ T = [T_{i,j}]_{i,j\in N}, S = [S_{i,j}]_{i,j\in N} \in \mathbf{B}(\ell^{2}(N,H)) $, then
\begin{equation}
    \label{FormAdjointOfMatrix}
    T^{\ast} = [T^{\ast}_{j,i}]_{i,j\in N}
\end{equation}
and
\begin{equation}
    \label{FormProductOfMatrices}
    TS = \left[ \sum_{k\in N} T_{i,k}S_{k,j} \right]_{i,j\in N}.
\end{equation}
If $ (S_{i})_{i\in \mathbb{Z}_{1}} \subset \mathbf{B}(H) $ is a uniformly bounded sequence of operators, then we define the unilateral weighted shift $ S\in \mathbf{B}(\ell^{2}(\mathbb{N},H)) $ with weights $ (S_{i})_{i\in \mathbb{Z}_{1}} $ as follows
\begin{equation*}
    S(h_{i})_{i\in \mathbb{N}} = (0, S_{1}h_{0}, S_{2}h_{1}, \ldots), \quad (h_{i})_{i\in \mathbb{N}}\in \ell^{2}(\mathbb{N},H).
\end{equation*}
It is a straightforward computation to prove that (cf. \cite[Lemma 2.2]{lambertUnitaryEquivalence1971}) 
\begin{equation}
    \label{FormNormOfUWS}
    \lVert S\rVert = \sup_{n\in \mathbb{Z}_{1}} \lVert S_{n}\rVert.
\end{equation}The matrix representation of $ S $ takes the form
\begin{equation*}
    S = \begin{bmatrix}
        0 & 0 & 0 & 0 & \ddots\\
        S_{1} & 0 & 0 & 0 & \ddots \\
        0 & S_{2} & 0 & 0 & \ddots\\
        0 & 0 & S_{3} & 0 & \ddots\\
        \ddots & \ddots & \ddots & \ddots & \ddots
    \end{bmatrix}.
\end{equation*}
Similarly, if $ (S_{i})_{i\in \mathbb{Z}} \subset \mathbf{B}(H) $ is a uniformly bounded sequence of operators, we define the bilateral weighted shift $ S\in \mathbf{B}(\ell^{2}(\mathbb{Z},H)) $ with weights $ (S_{i})_{i\in \mathbb{Z}} $ by the formula
\begin{equation*}
    S(h_{i})_{i\in \mathbb{Z}} = (S_{i}h_{i-1})_{i\in \mathbb{Z}}, \quad (h_{i})_{i\in \mathbb{Z}}\in \ell^{2}(\mathbb{Z},H);
\end{equation*}
the matrix representation of this operator takes the form
\begin{equation}
    \label{FormMatrixOfBWS}
    S = \begin{bmatrix}
        \ddots &  \ddots & \ddots & \ddots & \ddots\\
        \ddots & 0 & 0 & 0 & \ddots\\
        \ddots &  S_{0} & \boxed{0} & 0 & \ddots\\
        \ddots & 0 & S_{1} & 0 & \ddots\\
        \ddots & \ddots & \ddots & \ddots & \ddots
    \end{bmatrix},
\end{equation}
where $ \boxed{\,\cdot\,} $ denotes the entry in zeroth column and zeroth row of this matrix. As in the case of unilateral shifts,
\begin{equation}
    \label{FormNormOfBWS}
    \lVert S\rVert = \sup_{n\in \mathbb{Z}}\lVert S_{n}\rVert.
\end{equation} 
It turns out that the adjoint of a bilateral weighted shift is also (up to a unitary equivalence) a bilateral weighted shift.
\begin{lemma}
    \label{LemOnAdjointOfBWS}
    Let $ H $ be a Hilbert space. If $ S\in \mathbf{B}(\ell^{2}(\mathbb{Z}),H) $ is a bilateral weighted shift with weights $ (S_{j})_{j\in \mathbb{Z}}\subset \mathbf{B}(H) $, then $ S^{\ast} $ is unitarily equivalent to a bilateral weighted shift with weights $ (S_{-j+1}^{\ast})_{j\in \mathbb{Z}} $.
\end{lemma} 
\begin{proof}
    By \eqref{FormAdjointOfMatrix} and \eqref{FormMatrixOfBWS}, we have
    \begin{equation*}
        S^{\ast} = \begin{bmatrix}
            \ddots &  \ddots & \ddots & \ddots & \ddots\\
            \ddots & 0 & S_{0}^{\ast} & 0 & \ddots\\
            \ddots &  0 & \boxed{0} & S_{1}^{\ast} & \ddots\\
            \ddots & 0 & 0 & 0 & \ddots\\
            \ddots & \ddots & \ddots & \ddots & \ddots
        \end{bmatrix},
    \end{equation*}
    so $ S^{\ast}(x_{j})_{j\in \mathbb{Z}} = (S_{j+1}^{\ast}x_{j+1})_{j\in \mathbb{Z}} $ for $ (x_{j})_{j\in \mathbb{Z}} \in \ell^{2}(\mathbb{Z},H) $. Denote by $ \widehat{S} $ the bilateral weighted shift with weights $ (S_{-j+1}^{\ast})_{j\in \mathbb{Z}} $. Define $ U\in \mathbf{B}(\ell^{2}(\mathbb{Z},H)) $ by the formula:
    \begin{equation*}
        U(x_{j})_{j\in \mathbb{Z}} = (x_{-j})_{j\in \mathbb{Z}}, \quad (x_{j})_{j\in \mathbb{Z}} \in \ell^{2}(\mathbb{Z},H).
    \end{equation*}
    Clearly, $ U $ is a unitary operator. Moreover,
    \begin{equation*}
        \widehat{S}U(x_{j})_{j\in \mathbb{Z}} = \widehat{S}(x_{-j})_{j\in \mathbb{Z}} = (S_{-j+1}^{\ast}x_{-j+1})_{j\in \mathbb{Z}}, \quad (x_{j})_{j\in \mathbb{Z}}\in \ell^{2}(\mathbb{Z},H),
    \end{equation*}
    and
    \begin{equation*}
        US^{\ast}(x_{j})_{j\in \mathbb{Z}} = U(S_{j+1}^{\ast}x_{j+1})_{j\in \mathbb{Z}} = (S_{-j+1}^{\ast}x_{-j+1})_{j\in \mathbb{Z}}, \quad (x_{j})_{j\in \mathbb{Z}}\in \ell^{2}(\mathbb{Z},H),
    \end{equation*}
    so $ \widehat{S}U = US^{\ast} $. Therefore, $ U $ makes $ S^{\ast} $ and $ \widehat{S} $ unitarily equivalent.
\end{proof}
It can be easily seen that
\begin{equation}
    \label{FormWhenBWSIsInvertible}
    S \text{ is invertible} \iff S_{j} \text{ is invertible for every } j\in \mathbb{Z}.
\end{equation}
The proof of the following observation is a straightforward application of \eqref{FormAdjointOfMatrix} and \eqref{FormProductOfMatrices} (see also \cite[Lemma 2.1]{buchalaUnitarilyEquivalentBilateralWeightedShifts2024})
\begin{observation}
    \label{ObsNthPowerOfShift}
    Let $ H $ be a Hilbert space. 
    \begin{enumerate}
        \item If $ S\in \mathbf{B}(\ell^{2}(\mathbb{N}),H) $ is a unilateral weighted shift with weights $ (S_{i})_{i\in \mathbb{Z}_{1}}\subset \mathbf{B}(H) $, then for every $ n\in \mathbb{Z}_{1} $,
        \begin{equation}
            \label{FormNthPowerOfUWS}
            (S^{\ast n}S^{n})_{i,j} = \begin{cases}
                \lvert S_{i+n}\cdots S_{i+1}\rvert^{2}, & \text{if } i = j\\
                0, & \text{otherwise}
            \end{cases}, \quad i,j\in \mathbb{N},
        \end{equation}
        are the entries of the matrix representation of $ S^{\ast n}S^{n} $.
        \item If $ S\in \mathbf{B}(\ell^{2}(\mathbb{Z}),H) $ is a bilateral weighted shift with weights $ (S_{i})_{i\in \mathbb{Z}}\subset \mathbf{B}(H) $, then for every $ n\in \mathbb{Z}_{1} $ the entries of the matrix representation of $ S^{\ast n}S^{n} $ are given by \eqref{FormNthPowerOfUWS} for all $ i,j\in \mathbb{Z} $.
    \end{enumerate}
\end{observation}	

	\section{$ m $-isometric unilateral weighted shifts}
\label{SecMIsometricUWS}
In this section we are going to characterize $ m $-isometric unilateral weighted shifts with invertible weights. Let us state the main result of this section. It is a natural generalization of \cite[Theorem 1]{abdullahStructureMisometricWeighted2016}.
\begin{theorem}
    \label{ThmCharOfMIsometricUWSByOperatorPolynomial}
    Let $ H $ be a Hilbert space. Suppose $ S\in \mathbf{B}(\ell^{2}(\mathbb{N}),H) $ is a unilateral weighted shift with invertible weights $ (S_{j})_{j\in \mathbb{Z}_{1}} $. For $ m\in \mathbb{Z}_{1} $ the following conditions are equivalent:
    \begin{enumerate}
        \item $ S $ is $ m $-isometric,
        \item for every $ x\in H $, $ x\not=0 $, the classical weighted shift $ S_{x}\in \ell^{2}(\mathbb{N}, \mathbb{C}) $ with the weights $ (\alpha_{j}(x))_{j\in \mathbb{Z}_{1}}\subset (0,\infty) $ given by
        \begin{equation*}
            \alpha_{j}(x) = \begin{cases}
                \lVert S_{1}x\rVert, & j = 1\\
                \frac{\lVert S_{j}\cdots S_{1}x\rVert}{\lVert S_{j-1}\cdots S_{1}x\rVert}, & j\ge 2
            \end{cases},
        \end{equation*}
        is $ m $-isometric,
        \item there exists a polynomial $ p\in \mathbf{B}(H)_{m-1}[z] $ such that $ p(0) = I $ and
        \begin{equation}
            \label{FormUWSPolynomialAtNaturalNumbers}
            p(n) = \lvert S_{n}\cdots S_{1}\rvert^{2}, \quad n\in \mathbb{Z}_{1}.
        \end{equation}
    \end{enumerate}
    Moreover, if (iii) holds, then
    \begin{equation}
        \label{FormNormOfMIsometricUWS}
        \lVert S\rVert^{2} = \sup\left\{ \frac{\langle p(n+1)h,h\rangle}{\langle p(n)h,h\rangle}\!: n\in \mathbb{N}, h\in H, h\not=0 \right\}.
    \end{equation}
\end{theorem}
Before we prove this theorem, we need a technical lemma.
\begin{lemma}
    \label{LemOperatorPolynomialFromCharOfMIsometricUWS}
    Let $ H $ and $ S $ be as in Theorem \ref{ThmCharOfMIsometricUWSByOperatorPolynomial}. If $ p\in \mathbf{B}(H)_{m-1}[z] $ is a polynomial as in Theorem \ref{ThmCharOfMIsometricUWSByOperatorPolynomial}(iii), then every coefficient of $ p $ is a linear combination of operators $ I,\lvert S_{1}\rvert^{2},\ldots, \lvert S_{m-1}\cdots S_{1}\rvert^{2} $ with real coefficients. In particular, every coefficient of $ p $ is self-adjoint and for every $ x\in H $, $ \langle p(\,\cdot\,)x,x\rangle \in \mathbb{R}_{m-1}[z] $. 
\end{lemma}
\begin{proof}
    Write $ p(z) = A_{m-1}z^{m-1}+\ldots+A_{1}z+I $, where $ A_{m-1},\ldots, A_{1}\in \mathbf{B}(H) $. Then for every $ x\in H $,
    \begin{equation*}
        q_{x}(z) := \langle p(z)x,x\rangle = \sum_{j=0}^{m-1}\langle A_{j}x,x\rangle z^{j} \in \mathbb{C}_{m-1}[z].
    \end{equation*}
    By the assumption, the coefficients of $ q_{x} $ satisfy the following system of linear equations\footnote{We stick to the convention that $ 0^{0} = 1 $}:
    \begin{align*}
        \begin{bmatrix}
            0^{m-1} & \ldots & 0^{0}\\
            1^{m-1} &  \ldots & 1^{0}\\
            \vdots & \vdots & \vdots\\
            (m-1)^{m-1} & \ldots & (m-1)^{0}
        \end{bmatrix} \begin{bmatrix}
            \langle A_{m-1}x,x\rangle\\
            \langle A_{m-2}x,x\rangle\\
            \vdots\\
            \langle x,x\rangle
        \end{bmatrix} = \begin{bmatrix}
            \langle x,x\rangle\\
            \langle \lvert S_{1}\rvert^{2}x,x\rangle\\
            \vdots\\
            \langle \lvert S_{m-1}\cdots S_{1}\rvert^{2}x,x\rangle
        \end{bmatrix}.
    \end{align*}
    By the Cramer rule (see \cite[Theorem 3.36.(7)]{nairLinearAlgebra2018}), the only solution of the above system is of the form
    \begin{equation*}
        \langle A_{j}x,x\rangle = M_{j,0}\langle x,x\rangle + \sum_{\ell=1}^{m-1} M_{j,\ell}\langle \lvert S_{\ell}\cdots S_{1}\rvert^{2}x,x\rangle, \quad j=0,1,\ldots,m-1,
    \end{equation*}
    where the constants $ M_{j,\ell}\in \mathbb{R} $ depend only on $ j $ and $ \ell $ (not on $ x\in H $). Thus,
    \begin{equation*}
        A_{j} = M_{j,0}I+ \sum_{\ell=1}^{m-1}M_{j,\ell}\lvert S_{\ell}\cdots S_{1}\rvert^{2}, \quad j=0,\ldots,m-1.
    \end{equation*}
    This completes the proof.
\end{proof}
\begin{proof}[Proof of Theorem \ref{ThmCharOfMIsometricUWSByOperatorPolynomial}]
    (i)$ \iff $(ii). We infer from Observation \ref{ObsNthPowerOfShift} that the matrix of $ \beta_{m}(S) $ is diagonal and that the entries on the diagonal of this matrix are given by the formula
    \begin{equation*}
        (\beta_{m}(S))_{j,j} = \sum_{k=0}^{m}(-1)^{m-k}\binom{m}{k} \lvert S_{j+k}\cdots S_{j+1}\rvert^{2}, \quad j\in \mathbb{N}. 
    \end{equation*}
    Therefore, $ \beta_{m}(S) = 0 $ if and only if $ (\beta_{m}(S))_{j,j} = 0 $ for every $ j\in \mathbb{N} $. In turn, for $ x\in H $, $ x\not=0 $,
    \begin{align*}
        \langle (\beta_{m}(S))_{j,j}x,x\rangle &= \sum_{k=0}^{m}(-1)^{m-k}\binom{m}{k}\lVert S_{j+k}\cdots S_{j+1}x\rVert^{2}\\
        &=\lVert S_{j}\cdots S_{1}x\rVert^{2}\sum_{k=0}^{m}(-1)^{m-k}\binom{m}{k} \alpha_{j+k}(x)^{2}\cdots \alpha_{0}(x)^{2},
    \end{align*}
    where $ \alpha_{0}(x) = 1 $ for all $ x\in H $. In view of \cite[Proposition 3.2]{bermudezWeightedShiftsWhich2010}, $ (\beta_{m}(S))_{j,j} = 0 $ for every $ j\in \mathbb{N} $ if and only if $ S_{x} $ is $ m $-isometric for every $ x\not=0 $. 

    (i)$ \Longrightarrow $(iii). By \cite[Eq.(1.3)]{aglerMisometricTransformationsHilbert1995}, for every $ k\in \mathbb{N} $ we have
    \begin{equation}
        S^{\ast k}S^{k} = \sum_{n=0}^{m-1}\frac{1}{n!}\beta_{n}(S)k(k-1)\cdots(k-n+1).
    \end{equation}
    It follows from Observation \ref{ObsNthPowerOfShift} that
    \begin{equation*}
        \lvert S_{k}\cdots S_{1}\rvert^{2} = (S^{\ast k}S^{k})_{0,0} = \sum_{n=0}^{m-1}\frac{1}{n!}\beta_{n}(S)_{0,0}k(k-1)\cdots(k-n+1).
    \end{equation*}
    Writing $ k(k-1)\cdots(k-n+1) = \sum_{\ell=0}^{n}C_{n,\ell}k^{\ell} $ ($ n\in \mathbb{N} $) we derive from the above equality that
    \begin{equation*}
        \lvert S_{k}\cdots S_{1}\rvert^{2} = \sum_{n=0}^{m-1} \frac{1}{n!}\beta_{n}(S)_{0,0}\sum_{\ell=0}^{n}C_{n,\ell}k^{\ell}, \quad k\in \mathbb{N}.
    \end{equation*}
    Hence, the polynomial
    \begin{equation*}
        p(z) = \sum_{n=0}^{m-1} \frac{1}{n!}\beta_{n}(S)_{0,0}\sum_{\ell=0}^{n}C_{n,\ell}z^{\ell} \in \mathbf{B}(H)_{m-1}[z]
    \end{equation*}
    satisfies (iii).\\
    (iii)$ \Longrightarrow $(ii). From (iii) we deduce that for every $ x\in H $ and every $ n\in \mathbb{Z}_{1} $,
    \begin{equation}
        \lVert S_{n}\cdots S_{1}x\rVert^{2} = \langle \lvert S_{n}\cdots S_{1}\rvert^{2}x,x\rangle = \langle p(n)x,x\rangle.
    \end{equation}
    For $ x\in H $ set $ q_{x}(z) = \langle p(z)x,x\rangle $. Then, by Lemma \ref{LemOperatorPolynomialFromCharOfMIsometricUWS}, $ q\in \mathbb{R}_{m-1}[x] $ and for $ x\not=0 $ we have $ q(1) = \lVert S_{1}x\rVert^{2} $ and
    \begin{equation}
        \frac{\lVert S_{n+1}\cdots S_{1}x\rVert^{2}}{\lVert S_{n}\cdots S_{1}x\rVert^{2}} = \frac{q(n+1)}{q(n)}, \quad n\in \mathbb{Z}_{1}.
    \end{equation}
    Therefore, by \cite[Theorem 1]{abdullahStructureMisometricWeighted2016}, the classical weighted shift $ S_{x} $ is $ m $-isometric for every $ x\not=0 $. \\
    Now let us prove the ''moreover'' part. Since for every $ n\in \mathbb{Z}_{1} $, $ S_{n}\cdots S_{1} $ is invertible, we have $ \mathcal{R}(S_{n}\cdots S_{1}) = H $ and
    \begin{align*}
        \lVert S_{n+1}\rVert^{2} &= \sup\left\{ \frac{\lVert S_{n+1}h\rVert^{2}}{\lVert h\rVert^{2}}\colon h\in H,h\not=0 \right\} \\
        &= \sup\left\{ \frac{\lVert S_{n+1}\cdots S_{1}h\rVert^{2}}{\lVert S_{n}\cdots S_{1}h\rVert^{2}}\colon h\in H,h\not=0 \right\}\\
        &= \sup\left\{ \frac{\lVert \lvert S_{n+1}\cdots S_{1}\rvert h\rVert^{2}}{\lVert \lvert S_{n}\cdots S_{1}\rvert h\rVert^{2}}\colon h\in H,h\not=0 \right\}\\
        &= \sup\left\{ \frac{\langle \lvert S_{n+1}\cdots S_{1}\rvert^{2} h,h\rangle}{\langle \lvert S_{n}\cdots S_{1}\rvert^{2} h,h\rangle}\colon h\in H,h\not=0 \right\}\\
        &= \sup\left\{ \frac{\langle p(n+1) h,h\rangle}{\langle p(n) h,h\rangle}\colon h\in H,h\not=0 \right\}.
    \end{align*}
    Using the fact that $ p(0) = I $ and proceeding similarly as above, we also have
    \begin{equation*}
        \lVert S_{1}\rVert^{2} = \sup\left\{ \frac{\langle p(1)h,h\rangle}{\langle p(0)h,h\rangle}\colon h\in H,h\not=0 \right\}.
    \end{equation*}
    Therefore, \eqref{FormNormOfMIsometricUWS} follows from the above and \eqref{FormNormOfUWS}.
\end{proof}
In \cite[Theorem 1]{abdullahStructureMisometricWeighted2016} the authors actually showed that for a classical unilateral weighted shifts $ S\in \mathbf{B}(\ell^{2}(\mathbb{N}, \mathbb{C})) $ with scalar weights $ (\lambda_{n})_{n\in \mathbb{Z}_{1}}\subset \mathbb{C}\setminus\{0\} $, $ m $-isometricity of $ S $ is equivalent to the existence of a polynomial $ p\in \mathbb{R}_{m-1}[x] $ satisfying the following formula:
\begin{equation*}
    \lvert \lambda_{n}\rvert^{2} = \frac{p(n+1)}{p(n)}, \quad n\in \mathbb{Z}_{1}.
\end{equation*}
As a corollary from Theorem \ref{ThmCharOfMIsometricUWSByOperatorPolynomial}, we obtain the counterpart of the above formula for shifts with operator weights.
\begin{corollary}
    \label{CorFormulaForWeightsUWSDoublyCommutingWeights}
    Let $ H $ and $ S $ be as in Theorem \ref{ThmCharOfMIsometricUWSByOperatorPolynomial}.
    \begin{enumerate}
        \item If $ S $ is $ m $-isometric and $ p\in \mathbf{B}(H)_{m-1}[z] $ is a polynomial as in Theorem \ref{ThmCharOfMIsometricUWSByOperatorPolynomial}(iii), then for every $ n\in \mathbb{Z}_{1} $, $ \lvert S_{n}\rvert^{2} $ is unitarily equivalent to $ p(n-1)^{-1/2}p(n)p(n-1)^{-1/2} $.
        \item Assume that the weights of $ S $ are pairwise doubly commuting. Then $ S $ is $ m $-isometric if and only if there exists a polynomial $ p\in \mathbf{B}(H)_{m-1}[z] $ such that $ p(0) = I $ and
        \begin{equation}
            \label{FormUWSPolynomialAtNaturalNumbersDoublyCommutingWeights}
            \lvert S_{n}\rvert^{2} = p(n)p(n-1)^{-1} = p(n-1)^{-1}p(n), \quad n\in \mathbb{Z}_{1}.
        \end{equation}
    \end{enumerate}
\end{corollary}
\begin{proof}
    (i). For $ n\in \mathbb{Z}_{1} $ let $ S_{n}\cdots S_{1} = U_{n}\lvert S_{n}\cdots S_{1}\rvert $ be the polar decomposition of $ S_{n}\cdots S_{1} $. Since $ S_{n}\cdots S_{1} $ is invertible, $ U_{n} $ is unitary for every $ n\in \mathbb{Z}_{1} $. For $ n\in \mathbb{Z}_{2} $ we have
    \begin{align*}
        p(n)=\lvert S_{n}\cdots S_{1}\rvert^{2} &= (S_{n-1}\cdots S_{1})^{\ast}S_{n}^{\ast}S_{n}S_{n-1}\cdots S_{1}\\
        &= \lvert S_{n-1}\cdots S_{1}\rvert U_{n-1}^{\ast}\lvert S_{n}\rvert^{2}U_{n-1}\lvert S_{n-1}\cdots S_{1}\rvert\\
        &=p(n-1)^{1/2}U_{n-1}^{\ast}\lvert S_{n}\rvert^{2}U_{n-1}p(n-1)^{1/2}.
    \end{align*}
    This implies that
    \begin{equation*}
        U_{n-1}^{\ast}\lvert S_{n}\rvert^{2} = p(n-1)^{-1/2}p(n)p(n-1)^{-1/2}U_{n-1}^{\ast}, \quad n\in \mathbb{Z}_{2}.
    \end{equation*}
    Since $ p(0) = I $, we have
    \begin{equation*}
        \lvert S_{1}\rvert^{2} = p(1) = p(0)^{-1/2}p(1)p(0)^{-1/2}.
    \end{equation*}
    Therefore, $ \lvert S_{n}\rvert^{2} $ is unitarily equivalent to $ p(n-1)^{-1/2}p(n)p(n-1)^{-1/2} $ for every $ n\in \mathbb{Z}_{1} $.\\
    (ii). Since the weights of $ S $ are doubly commuting, we infer from \cite[Lemma 2.1]{anandCompleteSystemsOfUnitaryInvariants} that
    \begin{equation*}
        \lvert S_{n}\cdots S_{1}\rvert^{2} = \lvert S_{n}\rvert^{2}\cdots \lvert S_{1}\rvert^{2}, \quad n\in \mathbb{Z}_{1}.
    \end{equation*}
    Using the above equality, we can easily show that a polynomial $ p\in \mathbf{B}(H)_{m-1}[z] $ satisfies \eqref{FormUWSPolynomialAtNaturalNumbers} if and only if it satisfies \eqref{FormUWSPolynomialAtNaturalNumbersDoublyCommutingWeights}.
\end{proof}
Theorem \ref{ThmCharOfMIsometricUWSByOperatorPolynomial} enables us to give a procedure of constructing $ m $-isometric unilateral weighted shifts with positive and invertible weights; since, by \cite[Theorem 3.4]{lambertUnitaryEquivalence1971}, every unilateral weighted shift with invertible weights is unitarily equivalent to the shift with positive and invertible weights, the assumption of positivity of weights is not restricting.
\begin{theorem}
    \label{ThmConstructingMIsometricUWS}
    Let $ H $ be a Hilbert space. If $ p\in \mathbf{B}(H)_{m-1}[z] $ is a polynomial such that
    \begin{enumerate}
        \item $ p(0) = I $,
        \item $ p(n) $ is positive and invertible for every $ n\in \mathbb{N} $,
        \item $ C(p):= \sup\left\{ \frac{\langle p(n+1)h,h\rangle}{\langle p(n)h,h\rangle}\colon n\in \mathbb{N}, h\in H, h\not=0 \right\}<\infty $,
    \end{enumerate}
    then there exists an $ m $-isometric unilateral weighted shift $ S\in \mathbf{B}(\ell^{2}(\mathbb{N},H)) $ with positive and invertible weights $ (S_{j})_{j\in \mathbb{Z}_{1}}\subset \mathbf{B}(H) $ satisfying \eqref{FormUWSPolynomialAtNaturalNumbers}. Moreover, $ \lVert S\rVert = \sqrt{C(p)} $.
\end{theorem}
\begin{proof}
    We define the sequence $ (S_{j})_{j\in \mathbb{Z}_{1}} $ by induction. Set $ S_{1}:= p(1)^{1/2} $. Then $ S_{1} $ is positive and invertible and such that $ \langle \lvert S_{1}\rvert^{2}h,h\rangle = \langle p(1)h,h\rangle $ for every $ h\in H $. Assume now that for some $ k\in \mathbb{Z}_{1} $ we have defined the weights $ S_{1},\ldots, S_{k}\in \mathbf{B}(H) $ in a way that they are invertible and positive operators satisfying \eqref{FormUWSPolynomialAtNaturalNumbers} for $ n=1,\ldots,k $. Then
    \begin{equation*}
        \widetilde{S}_{k+1} := S_{1}^{-1}\cdots S_{k}^{-1}p(k+1)S_{k}^{-1}\cdots S_{1}^{-1}
    \end{equation*}
    is a positive and invertible operator and so is $ S_{k+1}:= \widetilde{S}_{k+1}^{1/2} $. We have
    \begin{equation*}
        \langle \lvert S_{k+1}\cdots S_{1}\rvert^{2}h,h\rangle = \langle S_{1}\cdots S_{k}\cdot \widetilde{S}_{k+1}\cdot S_{k}\cdots S_{1}h,h\rangle = \langle p(k+1)h,h\rangle, \quad h\in H,
    \end{equation*}
    so \eqref{FormUWSPolynomialAtNaturalNumbers} holds for $ n=k+1 $. It remains to show that the sequence $ (S_{j})_{j\in \mathbb{Z}_{1}}\subset \mathbf{B}(H) $ defined above is uniformly bounded. By (i) and (iii) we have
    \begin{equation*}
        \lVert S_{1}h\rVert^{2} = \langle S_{1}^{2}h,h\rangle = \langle p(1)h,h\rangle \le C(p)\cdot \langle p(0)h,h\rangle = C(p)\cdot \lVert h\rVert^{2}, \quad h\in H,
    \end{equation*}
    and
    \begin{align*}
        \lVert S_{n+1}\cdots S_{1}h\rVert^{2} &= \langle \lvert S_{n+1}\cdots S_{1}\rvert^{2}h,h\rangle = \langle p(n+1)h,h\rangle \\
        &\le C(p)\cdot \langle p(n)h,h\rangle = C(p)\cdot \lVert S_{n}\cdots S_{1}h\rVert^{2}, \quad n\in \mathbb{Z}_{1}, h\in H.
    \end{align*}
    Since $ \mathcal{R}(S_{n}\cdots S_{1}) = H $ for every $ n\in \mathbb{Z}_{1} $, it follows that $ \lVert S_{n}\rVert^{2} \le C(p) $ for every $ n\in \mathbb{Z}_{1} $. We infer from Theorem \ref{ThmCharOfMIsometricUWSByOperatorPolynomial} that the weighted shift $ S\in \mathbf{B}(\ell^{2}(\mathbb{N},H)) $ with weights $ (S_{j})_{j\in \mathbb{Z}_{1}} $ is $ m $-isometric. The formula $ \lVert S\rVert = \sqrt{C(p)} $ follows from \eqref{FormNormOfMIsometricUWS}.
\end{proof}
The following simple observation gives us another way of checking the condition Theorem~\ref{ThmConstructingMIsometricUWS}(iii).
\begin{observation}
    \label{ObsSupremumBySpectralRadius}
    Let $ H $ be a Hilbert space. If $ p\in \mathbf{B}(H)_{m-1}[z] $ is a polynomial satisfying Theorem \ref{ThmConstructingMIsometricUWS}(i)--(ii), then
    \begin{equation}
        \label{FormSupremumBySpectralRadius}
        C(p) = \sup\left\{ r(p(n+1)p(n)^{-1})\colon n\in \mathbb{N}\right\}.
    \end{equation}
\end{observation}
\begin{proof}
    By the fact that for every $ n\in \mathbb{N} $, $ p(n) $ is positive and invertible, we have
    \begin{align*}
        &\sup\left\{ \frac{\langle p(n+1)h,h\rangle}{\langle p(n)h,h\rangle}\colon n\in \mathbb{N}, h\in H,h\not=0 \right\} \\
        &= \sup\left\{ \frac{\langle p(n+1)p(n)^{-1/2}h,p(n)^{-1/2}h\rangle}{\langle p(n)p(n)^{-1/2}h,p(n)^{-1/2}h\rangle}\colon n\in \mathbb{N}, h\in H,h\not=0 \right\}\\
        &=\sup\left\{ \frac{\langle p(n)^{-1/2}p(n+1)p(n)^{-1/2}h,h\rangle}{\langle h,h\rangle}\colon n\in \mathbb{N}, h\in H,h\not=0 \right\}\\
        &= \sup\left\{ \lVert p(n)^{-1/2}p(n+1)p(n)^{-1/2}\rVert\colon n\in \mathbb{N} \right\}.
    \end{align*}
    Since $ p(n)^{-1/2}p(n+1)p(n)^{-1/2} $ is positive for every $ n\in \mathbb{N} $, it follows from \cite[Theorem 5.44]{weidmannLinearOperatorsInHilbertSpace} that
    \begin{equation*}
        \lVert p(n)^{-1/2}p(n+1)p(n)^{-1/2}\rVert = r(p(n)^{-1/2}p(n+1)p(n)^{-1/2}), \quad n\in \mathbb{N}.
    \end{equation*}
    It is easy to see that for $ \lambda\in \mathbb{C} $ and $ n\in \mathbb{N} $, $ p(n)^{-1/2}p(n+1)p(n)^{-1/2} -\lambda I $ is invertible if and only if so is $ p(n+1)-\lambda p(n) $. In turn, $ p(n+1)-\lambda p(n) $ is invertible if and only if so is $ p(n+1)p(n)^{-1}-\lambda I $. Hence,
    \begin{equation*}
        r(p(n)^{-1/2}p(n+1)p(n)^{-1/2}) = r(p(n+1)p(n)^{-1}), \quad n\in \mathbb{N},
    \end{equation*}
    and the proof is complete.
\end{proof}
The following example shows that in general the condition (iii) of Theorem \ref{ThmConstructingMIsometricUWS} cannot be removed even if $ H $ is finite dimensional.
\begin{example}
    Define $ p\in \mathbf{B}(\mathbb{C}^{3})_{3}[z] $ by the formula:
    \begin{equation*}
        p(z) = \begin{bmatrix}
            z^{3}+16z+1 & \frac{1}{2}z^{2} & 2z\\
            \frac{1}{2}z^{2} & \frac{1}{4}z+1 & 0\\
            2z & 0 & 1
        \end{bmatrix}, \quad z\in \mathbb{C}.
    \end{equation*}
    An easy computation shows that for every $ n\in \mathbb{N} $, $ p(n) $ is self-adjoint and all leading principal minors of $ p(n) $ are positive. By the Sylvester criterion (see \cite[Theorem 7.2.5]{hornMatrixAnalysis}), $ p(n) $ is positive for every $ n\in \mathbb{N} $. It follows that $ p(n) $ is invertible. Obviously, $ p(0) = I $. Now observe that (cf. \cite[Eq. (1)]{panEstimatingExtremalEigenvalues} and the proof of Observation \ref{ObsSupremumBySpectralRadius})
    \begin{align}
        \label{ProofFormSpectralRadiusEstimation}
        r(p(n+1)p(n)^{-1}) &= r(p(n)^{-1/2}p(n+1)p(n)^{-1/2}) \\
        \notag&\le \tr(p(n)^{-1/2}p(n+1)p(n)^{-1/2}) \\
        \notag&\le 3r(p(n)^{-1/2}p(n+1)p(n)^{-1/2}),\\
        \notag&=3r(p(n+1)p(n)^{-1}), \quad n\in \mathbb{N}.
    \end{align}
    By the properties of a trace of matrix,
    \begin{equation*}
        \tr(p(n)^{-1/2}p(n+1)p(n)^{-1/2}) = \tr(p(n+1)p(n)^{-1}), \quad n\in \mathbb{N}.
    \end{equation*}
    Simple calculations gives us that
    \begin{equation*}
        \tr(p(n+1)p(n)^{-1}) = \frac{n^{2}+208n+81}{4+65n}, \quad n\in \mathbb{N}.
    \end{equation*}
    Combining the above with \eqref{ProofFormSpectralRadiusEstimation}, we get that
    \begin{equation*}
        r(p(n+1)p(n)^{-1}) \xrightarrow{n\to \infty} \infty.
    \end{equation*}
    By Observation \ref{ObsSupremumBySpectralRadius},
    \begin{equation*}
        \sup\left\{ \frac{\langle p(n+1)h,h\rangle}{\langle p(n)h,h\rangle}\colon n\in \mathbb{N}, h\in H, h\not=0 \right\} = \infty.
    \end{equation*}
\end{example}
However, under further assumption that the coefficients of $ p $ are commuting, the condition (iii) of Theorem \ref{ThmConstructingMIsometricUWS} is automatically satisfied.
\begin{lemma}
    \label{LemSupremumIsFiniteIfCoeffsCommute}
    Let $ H = \mathbb{C}^{N} $ ($ N\in \mathbb{Z}_{1} $) and let $ p\in \mathbf{B}(H)_{m-1}[z] $ be a polynomial satisfying Theorem \ref{ThmConstructingMIsometricUWS}(i)--(ii). If the coefficients of $ p $ are commuting, then the condition (iii) of Theorem \ref{ThmConstructingMIsometricUWS} holds.
\end{lemma}
\begin{proof}
    Write $ p(z) = A_{m-1}z^{m-1}+\ldots+A_{1}z+I $. From Observation \ref{LemOperatorPolynomialFromCharOfMIsometricUWS} it follows that $ A_{m-1},\ldots,A_{1} $ are self-adjoint. By \cite[Theorem 2.5.5]{hornMatrixAnalysis} there is no loss in generality in assuming that $ A_{m-1},\ldots,A_{1} $ are diagonal matrices. This implies that $ p $ is of the form
    \begin{equation*}
        p(z) = \begin{bmatrix}
            p_{1}(z) & 0 & \ldots & 0\\
            0 & p_{2}(z) & \ldots & 0\\
            \vdots & \vdots & \vdots & \vdots\\
            0 & 0 & \ldots & p_{N}(z)
        \end{bmatrix}, \quad z\in \mathbb{C},
    \end{equation*}
    where $ p_{j}\in \mathbb{R}_{m-1}[z] $ are such that $ p_{j}(n)> 0 $ ($ n\in \mathbb{N}, j = 1,\ldots,N $). We have
    \begin{equation*}
        p(n)^{-1/2}p(n+1)p(n)^{-1/2} = \begin{bmatrix}
            \frac{p_{1}(n+1)}{p_{1}(n)} & 0 & \ldots & 0\\
            0 & \frac{p_{2}(n+1)}{p_{2}(n)} & \ldots & 0\\
            \vdots & \vdots & \vdots & \vdots\\
            0 & 0 & \ldots & \frac{p_{N}(n+1)}{p_{N}(n)}
        \end{bmatrix}, \quad n\in \mathbb{N}.
    \end{equation*}
    Hence,
    \begin{align*}
        r(p(n)^{-1/2}p(n+1)p(n)^{-1/2}) &\le \tr(p(n)^{-1/2}p(n+1)p(n)^{-1/2}) \\
        &= \sum_{j=1}^{N}\frac{p_{j}(n+1)}{p_{j}(n)} \xrightarrow{n\to\infty} N.
    \end{align*}
    By Observation \ref{ObsSupremumBySpectralRadius}, the condition Theorem \ref{ThmConstructingMIsometricUWS}(iii) is satisfied.
\end{proof}
If $ H $ is infinite dimensional, then the conclusion of Lemma \ref{LemSupremumIsFiniteIfCoeffsCommute} does not hold.
\begin{example}
    For $ k\in \mathbb{Z}_{1} $ define a polynomial $ q_{k}\in \mathbb{R}_{2}[z] $ by the formula
    \begin{equation*}
        q_{k}(z) = \left( k+\frac{1}{k+1} \right)^{-1}\left( k+\frac{2}{3} \right)^{-1}\left( z-k-\frac{1}{k+1} \right)\left( z-k-\frac{2}{3} \right), \quad z\in \mathbb{C}.
    \end{equation*}
    Let $ H $ be an infinite dimensional separable Hilbert space and denote by $ (e_{j})_{j\in \mathbb{Z}_{1}}\subset H $ its orthonormal basis. For $ z\in \mathbb{C} $ define a diagonal operator $ p(z)\in \mathbf{B}(H) $ as follows:
    \begin{equation*}
        p(z)e_{k} = q_{k}(z)e_{k}, \quad k\in \mathbb{Z}_{1}.
    \end{equation*}
    Clearly, $ p\in \mathbf{B}(H)_{2}[z] $. Since $ q_{k}(n) > 0 $ for every $ n\in \mathbb{N} $, $ p(n) $ is a positive operator. It is easy to see that $ p(0) = I $. For every $ n\in \mathbb{N} $ we have
    \begin{equation*}
        \langle p(n)e_{k},e_{k}\rangle = q_{k}(n) = \left( \frac{k+1}{k^{2}+k+1}n-1 \right)\left( \frac{3}{3k+2}n-1 \right) \xrightarrow{k\to\infty} 1.
    \end{equation*}
    From the above equality we deduce that for every $ n\in \mathbb{N} $ there exists a constant $ c_{n}\in (0,\infty) $ such that $ \lVert p(n)h\rVert\ge c_{n}\lVert h\rVert $, $ h\in H $. This implies that $ \mathcal{R}(p(n)) $ is closed for every $ n\in \mathbb{N} $. Since $ p(n) $ is positive, it follows from \eqref{FormRangeKernelDecom} that $ p(n) $ is invertible for each $ n\in \mathbb{N} $. However,
    \begin{align*}
        \frac{\langle p(n+1)e_{n},e_{n}\rangle}{\langle p(n)e_{n},e_{n}\rangle} = \frac{\frac{1}{3}\left( 1-\frac{1}{n+1} \right)}{\frac{1}{n+1}\cdot \frac{2}{3}} = \frac{n}{2} \xrightarrow{n\to \infty} \infty.
    \end{align*}
    Therefore, the condition (iii) in Theorem \ref{ThmConstructingMIsometricUWS} is not satisfied.
\end{example}
As an another example of use of Theorem \ref{ThmCharOfMIsometricUWSByOperatorPolynomial} we prove that every sequence $ A_{1},\ldots,A_{m}\in \mathbf{B}(H) $ of invertible operators can be completed to the sequence of weights of $ (m+2) $-isometric unilateral weighted shift. Abdullah and Le proved such a result for classical unilateral weighted shifts with scalar weights (see \cite[Proposition 7]{abdullahStructureMisometricWeighted2016}); the idea of the proof for shifts with operator weights is similar. Before we state the result, we introduce some notation. For distinct complex numbers $ x_{0},\ldots,x_{m}\in \mathbb{C} $ we define the polynomials $ \ell_{j}\in \mathbb{C}[x] $ ($ j=0,\ldots,m $) as follows\footnote{We stick to the convention that $ \prod \varnothing = 1 $.}:
\begin{equation*}
    \ell_{j}(z) = \prod_{\substack{k=0,\ldots,m\\k\not=j}} \frac{z-x_{k}}{x_{j}-x_{k}}, \quad z\in \mathbb{C}.
\end{equation*}
It is a matter of routine to verify that $ \ell_{j}(x_{j}) = 1 $ for $ j=0,\ldots,m $ and $ \ell_{j}(x_{k}) = 0 $ for $ j,k\in \{0,\ldots,m\} $, $ j\not=k $.   
\begin{theorem}
    \label{ThmCompletingWeightsToMIsometricUWS}
    Let $ H $ be a Hilbert space and let $ A_{1},\ldots,A_{m} \in \mathbf{B}(H) $ be invertible operators. Then there exists an $ (m+2) $-isometric unilateral weighted shift $ S\in \mathbf{B}(\ell^{2}(\mathbb{N},H)) $ with weights $ (S_{j})_{j\in \mathbb{Z}_{1}} $ such that $ S_{j} = A_{j} $ for every $ j = 1,\ldots,m $ and that $ S_{j} $ is positive for $ j\in \mathbb{Z}_{m+1} $. 
\end{theorem}
\begin{proof}
    Let $ \ell_{j} $ ($ j=0,\ldots,m $) be the polynomials defined as above for the numbers $ x_{j} = j $. For $ h\in H $ set
    \begin{equation*}
        p(z) := \sum_{j=0}^{m} \widetilde{A}_{j}\ell_{j}(z), \quad z\in \mathbb{C},
    \end{equation*}
    where $ \widetilde{A}_{0} = I $ and $ \widetilde{A}_{j} = \lvert A_{j}\cdots A_{1}\rvert^{2} $ for $ j= 1,\ldots,m $.
    Then $ p\in \mathbf{B}(H)_{m}[z] $. Define $ q\in \mathbf{B}(H)_{m+1}[z] $ by the formula
    \begin{equation*}
        q(z) = \alpha z(z-1)\cdots(z-m)I+p(z), \quad z\in \mathbb{C}, h\in H,
    \end{equation*}
    where\footnote{Since $ \sum_{j=0}^{m} \ell_{j}(n) $ is the polynomial in $ n $ of degree $ m $ and $ n(n-1)\cdots(n-m) $ is the polynomial in $ n $ of degree $ m+1 $, the supremum on the right hand side is finite.}
    \begin{equation}
        \label{ProofFormConditionOnAlpha}
        \alpha > \sup_{n\ge m+1} \frac{\max_{j=0,\ldots,m}\lVert \widetilde{A}_{j}\rVert \sum_{j=0}^{m}\lvert \ell_{j}(n)\rvert}{n(n-1)\cdots(n-m)}.
    \end{equation}
    We have $ q(0) = p(0) = I $ and $ q(n) = p(n) = \lvert A_{n}\cdots A_{1}\rvert^{2} $ for $ n = 1,\ldots,m $. We will prove that $ q(n) $ is positive for $ n\in \mathbb{Z}_{m+1} $. If $ n\in \mathbb{Z}_{m+1} $ and $ h\in H $ is of norm one, then
    \begin{align*}
        \langle q(n)h,h\rangle &= \alpha n(n-1)\cdots (n-m)+ \sum_{j=0}^{m}\langle \widetilde{A}_{j}h,h\rangle\ell_{j}(n) \\
        &\ge \alpha n(n-1)\cdots (n-m)- \left\lvert\sum_{j=0}^{m}\langle \widetilde{A}_{j}h,h\rangle\ell_{j}(n)\right\rvert \\
        &\ge \alpha n(n-1)\cdots (n-m) - \sum_{j=0}^{m}\lvert \langle \widetilde{A_{j}}h,h\rangle\rvert \lvert \ell_{j}(n)\rvert\\
        &\ge \alpha n(n-1)\cdots (n-m) - \sum_{j=0}^{m} \lVert \widetilde{A_{j}}\rVert \lvert \ell_{j}(n)\rvert\\
        &\ge \alpha n(n-1)\cdots (n-m) - \max_{j=0,\ldots,m}\lVert \widetilde{A_{j}}\rVert \sum_{j=0}^{m} \lvert \ell_{j}(n)\rvert.
    \end{align*}
    By \eqref{ProofFormConditionOnAlpha},
    \begin{equation}
        \label{ProofFormPolynomialBoundedFromBelow}
        \langle q(n)h,h\rangle\ge \alpha n(n-1)\cdots (n-m) - \max_{j=0,\ldots,m}\lVert \widetilde{A_{j}}\rVert \sum_{j=0}^{m} \lvert \ell_{j}(n)\rvert >0
    \end{equation}
    From the above we also deduce that for every $ n\in \mathbb{Z}_{m+1} $ there exists $ c_{n}\in (0,\infty) $ such that $ \langle q(n)h,h\rangle \ge c_{n}\lVert h\rVert^{2} $ for every $ h\in H $. Therefore, $ \mathcal{R}(q(n)) $ is closed and, consequently, by \eqref{FormRangeKernelDecom}, $ q(n) $ is invertible for every $ n\in \mathbb{Z}_{m+1} $. We define the sequence $ (S_{j})_{j\in \mathbb{Z}_{1}} $ of invertible operators by induction. For $ j=1,\ldots, m $ set $ S_{j} = A_{j} $. If $ S_{1},\ldots, S_{n}\in \mathbf{B}(H) $ ($ n\in \mathbb{Z}_{m} $) are defined, then we set
    \begin{equation*}
        S_{n+1} := \left[ (S_{n}^{\ast})^{-1}\cdots (S_{1}^{\ast})^{-1}q(n+1)S_{n}^{-1}\cdots S_{1}^{-1} \right]^{1/2}.
    \end{equation*}
    Clearly, each $ S_{n} $ ($ n\ge m+1 $) is positive and invertible. Moreover,
    \begin{equation}
        \label{ProofFormModuliByPolynomial}
        \langle \lvert S_{n}\cdots S_{1}\rvert^{2}h,h\rangle = \langle q(n)h,h\rangle, \quad n\in \mathbb{Z}_{1}, h\in H.
    \end{equation}
    It remains to check that $ (S_{j})_{j\in \mathbb{Z}_{1}} $ is uniformly bounded.
    Note that for every $ n\in \mathbb{Z}_{m+1} $ and $ h\in H $ of norm one, we have
    \begin{align*}
        &\frac{\lVert S_{n+1}\cdots S_{n}h\rVert^{2}}{\lVert S_{n}\cdots S_{n}h\rVert^{2}} = \frac{\langle q(n+1)h,h\rangle}{\langle q(n)h,h\rangle}\\
        &= \frac{\alpha(n+1)n\cdots(n+1-m)+\sum_{j=0}^{m}\langle \widetilde{A_{j}}h,h\rangle\ell_{j}(n+1)}{\alpha n(n-1)\cdots(n-m)+\sum_{j=0}^{m}\langle \widetilde{A_{j}}h,h\rangle\ell_{j}(n)}\\
        &\stackrel{\eqref{ProofFormPolynomialBoundedFromBelow}}{\le}\frac{\alpha(n+1)n\cdots(n+1-m)+\max_{j=0,\ldots,m}\lVert \widetilde{A_{j}}\rVert\sum_{j=0}^{m}\ell_{j}(n+1)}{\alpha n(n-1)\cdots (n-m) - \max_{j=0,\ldots,m}\lVert \widetilde{A_{j}}\rVert \sum_{j=0}^{m} \lvert \ell_{j}(n)\rvert}.
    \end{align*}
    Since the polynomial $ \ell_{j} $ is of degree $ m $ for every $ j=0,1\ldots,m $, it is a matter of routine to verify that
    \begin{equation*}
        \lim_{n\to\infty} \frac{\alpha(n+1)n\cdots(n+1-m)+\max_{j=0,\ldots,m}\lVert \widetilde{A_{j}}\rVert\sum_{j=0}^{m}\ell_{j}(n+1)}{\alpha n(n-1)\cdots (n-m) - \max_{j=0,\ldots,m}\lVert \widetilde{A_{j}}\rVert \sum_{j=0}^{m} \lvert \ell_{j}(n)\rvert} = 1.
    \end{equation*}
    It follows that there exists $ C\in (0,\infty) $ such that for every $ n\in \mathbb{Z}_{m+1} $,
    \begin{equation*}
        \frac{\alpha(n+1)n\cdots(n+1-m)+\max_{j=0,\ldots,m}\lVert \widetilde{A_{j}}\rVert\sum_{j=0}^{m}\ell_{j}(n+1)}{\alpha n(n-1)\cdots (n-m) - \max_{j=0,\ldots,m}\lVert \widetilde{A_{j}}\rVert \sum_{j=0}^{m} \lvert \ell_{j}(n)\rvert} \le C.
    \end{equation*}
    Thus,
    \begin{equation*}
        \lVert S_{n+1}\cdots S_{n}h\rVert^{2} \le C \lVert S_{n}\cdots S_{n}h\rVert^{2}, \quad n\in \mathbb{Z}_{m+1}, h\in H.
    \end{equation*}
    Since $ \mathcal{R}(S_{n}\cdots S_{1}) = H $, we have
    \begin{equation*}
        \lVert S_{n+1}\rVert \le \sqrt{C}, \quad n\in \mathbb{Z}_{m+1}.
    \end{equation*}
    Therefore,
    \begin{equation*}
        \sup_{n\in \mathbb{Z}_{1}}\lVert S_{n}\rVert \le \max\{\sqrt{C}, \lVert S_{1}\rVert,\ldots,\lVert S_{m+1}\rVert \}.
    \end{equation*}
    Using \eqref{ProofFormModuliByPolynomial} we deduce from Theorem \ref{ThmCharOfMIsometricUWSByOperatorPolynomial} that the weighted shift $ S\in \mathbf{B}(\ell^{2}(\mathbb{Z},H)) $ with weights $ (S_{j})_{j\in \mathbb{Z}_{1}} $ is $ (m+2) $-isometric and such that $ S_{j} = A_{j} $ for every $ j=1,\ldots,m $.
\end{proof}
Theorem \ref{ThmCompletingWeightsToMIsometricUWS} states (in particular) that if $ A_{1}\in \mathbf{B}(H) $ is invertible, then there exists a 3-isometric unilateral weighted shift $ S\in \mathbf{B}(\ell^{2}(\mathbb{N},H)) $ with weights starting from $ A_{1} $. However, it may not be possible to find such a shift, which is 2-isometric. It was proved by Jabłoński in \cite[Proposition 3.2 and Remark 3.4]{jablosnkiHyperexpansiveOperatorValuedUWS} that $ A_{1} $ is a first weight of 2-isometric unilateral weighted shifts if and only if $ A_{1}^{\ast}A_{1}\ge I $. \par 
Next, let us make some remarks about commutativity of weights of $ m $-isometric unilateral weighted shift. We start from general result.
\begin{theorem}
    \label{ThmWhenWeightsCommuteUWS}
    Let $ H $ be a Hilbert space and let $ S\in \mathbf{B}(\ell^{2}(\mathbb{N},H)) $ be an $ m $-isometric unilateral weighted shifts with positive and invertible weights $ (S_{j})_{j\in \mathbb{Z}_{1}}\subset \mathbf{B}(H) $. The following conditions are equivalent:
    \begin{enumerate}
        \item all weights of $ S $ commute,
        \item $ S_{1},\ldots,S_{m-1} $ commute,
        \item if $ p(z) = A_{m-1}z^{m-1}+\ldots+A_{1}z+I \in \mathbf{B}(H)_{m-1}[z] $ is a polynomial as in Theorem \ref{ThmCharOfMIsometricUWSByOperatorPolynomial}(iii), then $ A_{m-1},\ldots,A_{1} $ commute. 
    \end{enumerate}
\end{theorem}
In the proof we use the following lemma.
\begin{lemma}
    \label{LemCommutingModuli}
    Let $ H $ be a Hilbert space and let $ (S_{j})_{n\in \mathbb{Z}_{1}}\subset \mathbf{B}(H) $ be a sequence of positive and invertible operators. Assume that for every $ n,m\in \mathbb{Z}_{1} $, $ \lvert S_{n}\cdots S_{1}\rvert $ commutes with $ \lvert S_{m}\cdots S_{1}\rvert $. Then $ S_{n} $ commutes with $ S_{m} $ for every $ n,m\in \mathbb{Z}_{1} $.
\end{lemma}
\begin{proof}
    We prove by induction that for every $ n\in \mathbb{Z}_{1} $ the operators $ S_{1},\ldots, S_{n} $ are commuting. First consider the case $ n=2 $. By the assumption, we have
    \begin{equation*}
        S_{1}S_{1}S_{2}S_{2}S_{1} = S_{1}\lvert S_{2}S_{1}\rvert^{2} = \lvert S_{2}S_{1}\rvert^{2}S_{1} = S_{1}S_{2}S_{2}S_{1}S_{1}.
    \end{equation*}
    Since $ S_{1} $ is invertible, it follows that
    \begin{equation*}
        S_{1}S_{2}^{2} = S_{2}^{2}S_{1}.
    \end{equation*}
    By the square root theorem (see \cite[p.265]{rieszFunctionalAnalysis}), $ S_{1} $ commutes with $ S_{2} $. Now suppose that for some $ n\in \mathbb{Z}_{2} $, $ S_{1},\ldots,S_{n} $ commute. We will show that
    \begin{equation}
        \label{ProofFormOneWeightCommuteWithModuliOfProduct}
        S_{k}\lvert S_{n+1}\cdots S_{1}\rvert = \lvert S_{n+1}\cdots S_{1}\rvert S_{k}, \quad k=1,\ldots,n.
    \end{equation}
    For $ k= 1 $, the above equality follows directly from the assumption. Let $ 2\le k \le n $. It is enough to prove that
    \begin{equation}
        \label{ProofFormOneWeightCommuteWithModuliOfProductInverse}
        S_{k}^{-2}\lvert S_{n+1}\cdots S_{1}\rvert^{2} = \lvert S_{n+1}\cdots S_{1}\rvert^{2}S_{k}^{-2}.
    \end{equation}
    Indeed, \eqref{ProofFormOneWeightCommuteWithModuliOfProductInverse} implies that
    \begin{equation*}
        S_{k}^{2}\lvert S_{n+1}\cdots S_{1}\rvert^{2} = \lvert S_{n+1}\cdots S_{1}\rvert^{2}S_{k}^{2}.
    \end{equation*}
    Using the square root theorem we obtain \eqref{ProofFormOneWeightCommuteWithModuliOfProduct}. For the proof of \eqref{ProofFormOneWeightCommuteWithModuliOfProductInverse} observe that, by \cite[Lemma 2.1]{anandCompleteSystemsOfUnitaryInvariants} and by the inductive hypothesis, we have
    \begin{align*}
        S_{k}^{-2}\lvert S_{n+1}\cdots S_{1}\rvert^{2} S_{k-1}^{-2}\cdots S_{1}^{-2} &= S_{k}^{-2}\lvert S_{n+1}\cdots S_{1}\rvert^{2}\lvert S_{k-1}\cdots S_{1}\rvert^{-2}\\
        &=S_{k}^{-2}\lvert S_{k-1}\cdots S_{1}\rvert^{-2}\lvert S_{n+1}\cdots S_{1}\rvert^{2}\\
        &=S_{k}^{-2}S_{k-1}^{-2}\cdots S_{1}^{-2}\lvert S_{n+1}\cdots S_{1}\rvert^{2}\\
        &=\lvert S_{k}\cdots S_{1}\rvert^{-2}\lvert S_{n+1}\cdots S_{1}\rvert^{2}\\
        &=\lvert S_{n+1}\cdots S_{1}\rvert^{2}\lvert S_{k}\cdots S_{1}\rvert^{-2}\\
        &=\lvert S_{n+1}\cdots S_{1}\rvert^{2} S_{k}^{-2}S_{k-1}^{-2}\cdots S_{1}^{-2}.
    \end{align*}
    Since $ \mathcal{R}(S_{k-1}^{-2}\cdots S_{1}^{-1}) = H $, \eqref{ProofFormOneWeightCommuteWithModuliOfProductInverse} follows from the above. Now, again by \cite[Lemma 2.1]{anandCompleteSystemsOfUnitaryInvariants},
    \begin{align*}
        &\lvert S_{n}\cdots S_{1}\rvert^{-1}\lvert S_{n+1}\cdots S_{1}\rvert^{2}\lvert S_{n}\cdots S_{1}\rvert^{-1} = S_{1}^{-1}\ldots S_{n}^{-1}\lvert S_{n+1}\cdots S_{1}\rvert^{2} S_{1}^{-1}\ldots S_{n}^{-1} \\
        &=S_{n}^{-1}\ldots S_{1}^{-1} S_{1}\ldots S_{n}S_{n+1}^{2}S_{n}\ldots S_{1}S_{1}^{-1}\ldots S_{n}^{-1} = S_{n+1}^{2}.
    \end{align*}
    By \eqref{ProofFormOneWeightCommuteWithModuliOfProduct}, we have
    \begin{align*}
        S_{k}S_{n+1}^{2} &= S_{k}\lvert S_{n}\cdots S_{1}\rvert^{-1}\lvert S_{n+1}\cdots S_{1}\rvert^{2}\lvert S_{n}\cdots S_{1}\rvert^{-1}\\
        &=\lvert S_{n}\cdots S_{1}\rvert^{-1}\lvert S_{n+1}\cdots S_{1}\rvert^{2}\lvert S_{n}\cdots S_{1}\rvert^{-1}S_{k}\\
        &=S_{n+1}^{2}S_{k}.
    \end{align*}
    By the square root theorem, $ S_{k}S_{n+1} = S_{n+1}S_{k} $. This completes the proof. 
\end{proof}
\begin{proof}[Proof of Theorem \ref{ThmWhenWeightsCommuteUWS}]
    The implication (i)$ \Longrightarrow $(ii) obviously holds; the implication (ii)$ \Longrightarrow $(iii) follows directly from Lemma \ref{LemOperatorPolynomialFromCharOfMIsometricUWS}. For the proof of implication (iii)$ \Longrightarrow $(i) note that for every $ n,m\in \mathbb{Z}_{1} $, by (iii), we have
    \begin{equation*}
        \lvert S_{n}\cdots S_{1}\rvert^{2}\lvert S_{n}\cdots S_{1}\rvert^{2} = p(n)p(m) = p(m)(n) = \lvert S_{m}\cdots S_{1}\rvert^{2}\lvert S_{n}\cdots S_{1}\rvert^{2}.
    \end{equation*} 
    The application of Lemma \ref{LemCommutingModuli} completes the proof.
\end{proof}

\begin{corollary}
    Let $ H $ be a Hilbert space and let $ S\in \mathbf{B}(\ell^{2}(\mathbb{N},H)) $ be an $ m $-isometric unilateral weighted shifts with positive and invertible weights such that the polynomial $ p\in \mathbf{B}_{m-1}(H)[z] $ as in Theorem \ref{ThmCharOfMIsometricUWSByOperatorPolynomial}(iii) is of the form $ p(z) = Az^{m-1}+I $ for some $ A\in \mathbf{B}(H) $. Then the weights of $ S $ commute. In particular, if $ S $ is 2-isometric, then the weights of $ S $ commute.
\end{corollary}
The fact that the weights of 2-isometric unilateral weighted shifts with positive and invertible weights are commuting was pointed out by Jabłoński in \cite[Theorem 3.3]{jablosnkiHyperexpansiveOperatorValuedUWS}; however, the proof is different from ours.

The assumption about positivity of weights in the above corollary cannot be dropped if $ \dim H \ge 2 $. Indeed, if $ U_{1}, U_{2}\in \mathbf{B}(H) $ are two non-commuting unitary operators, then the weighted shift $ S\in \mathbf{B}(\ell^{2}(\mathbb{N},H)) $ with the weights $ (S_{j})_{j\in \mathbb{Z}_{1}} $ defined by $ S_{j} = U_{j\mod 2} $ ($ j\in \mathbb{Z}_{1} $) is isometric (hence, 2-isometric) and, clearly, the weights of $ S $ do not commute. \par
The following example show that for every $ m\in \mathbb{Z}_{3} $ we can find $ m $-isometric weighted shift with positive, invertible and non-commuting weights.
\begin{example}
    Let $ H $ be a Hilbert space of dimension greater than one. Assume that $ A_{1},\ldots,A_{m}\subset \mathbf{B}(H) $ ($ m\in \mathbb{Z}_{2} $) are positive and invertible operators and at least two of them do not commute. Set $ A_{0} = I $. Define a polynomial $ p\in \mathbf{B}(H)_{m}[z] $ by the formula:
    \begin{equation*}
        p(z) = \sum_{j=0}^{m}A_{j}z^{j}, \quad z\in \mathbb{C}.
    \end{equation*}
    It is easy to see that $ p(n) $ is positive and invertible for every $ n\in \mathbb{Z}_{1} $ and that $ p(0) = I $. Since each $ A_{j} $ is positive and invertible, there exist $ c,C\in (0,\infty) $ such that
    \begin{equation*}
        C\lVert h\rVert^{2} \ge \langle A_{j}h,h\rangle = \lVert A_{j}^{1/2}h\rVert^{2} \ge c \lVert h\rVert^{2}, \quad h\in H, j=0,1,\ldots,m.
    \end{equation*}
    Hence,
    \begin{align*}
        \frac{\langle p(n+1)h,h\rangle}{\langle p(n)h,h\rangle} \le \frac{C\sum_{j=0}^{m}(n+1)^{j}}{c\sum_{j=0}^{m}n^{j}}, \quad n\in \mathbb{N}, h\in H, h\not=0.
    \end{align*}
    Since
    \begin{equation*}
        \frac{C\sum_{j=0}^{m}(n+1)^{j}}{c\sum_{j=0}^{m}n^{j}} \stackrel{n\to \infty}{\longrightarrow} \frac{C}{c},
    \end{equation*}
    the condition (iii) of Theorem \ref{ThmConstructingMIsometricUWS} is satisfied. By Theorem \ref{ThmConstructingMIsometricUWS} there exists $ (m+1) $-isometric weighted shift $ S\in \mathbf{B}(\ell^{2}(\mathbb{Z},H)) $ with positive and invertible weights $ (S_{j})_{j\in \mathbb{Z}_{1}} $ such that
    \begin{equation*}
        \lvert S_{n}\cdots S_{1}\rvert^{2} = \sum_{j=0}^{m}A_{j}n^{j}, \qquad n\in \mathbb{N}.
    \end{equation*}
    By Theorem \ref{ThmWhenWeightsCommuteUWS}, the weights of $ S $ do not commute.
\end{example}

	\section{$ m $-isometric bilateral weighted shifts}
\label{SecMIsometricBWS}
The aim of this section is to provide counterpart of results presented in Section \ref{SecMIsometricUWS} for bilateral weighted shifts. Let us introduce some notation.  If $ H $ is a Hilbert space, then for $ k\in \mathbb{Z} $ we denote
\begin{equation*}
    L_{k} = \{(h_{j})_{j\in \mathbb{Z}}\in \ell^{2}(\mathbb{Z},H)\colon h_{j} = 0 \text{ for } j<k\}.
\end{equation*}
If $ S\in \mathbf{B}(\ell^{2}(\mathbb{Z},H)) $ is a bilateral weighted shift with weights $ (S_{j})_{j\in \mathbb{Z}}\subset \mathbf{B}(H) $, then by $ S_{k} $ we denote the restriction of $ S $ to the subspace $ L_{k} $. Below we gather basic properties of the operators $ S_{k} $; we leave the proof to the reader.
\begin{observation}
    \label{ObsOnOperatorsSk}
    Let $ H $ be a Hilbert space and let $ S\in \mathbf{B}(\ell^{2}(\mathbb{Z},H)) $ be a bilateral weighted shift with weights $ (S_{j})_{j\in \mathbb{Z}}\subset \mathbf{B}(H) $. Then:
    \begin{enumerate}
        \item for every $ k\in \mathbb{Z} $, the subspace $ L_{k} $ is an invariant subspace for $ S $ and for $ S^{\ast n}S^{n} $ ($ n\in \mathbb{Z}_{1} $),
        \item for every $ k\in \mathbb{Z} $, $ S_{k} $ is unitarily equivalent to the unilateral weighted shift with weights $ (S_{j+k})_{j\in \mathbb{Z}_{1}} $,
        \item for $ m\in \mathbb{Z}_{1} $, $ S $ is $ m $-isometric if and only if $ S_{k} $ is $ m $-isometric for every $ k\in \mathbb{Z} $.
    \end{enumerate}
\end{observation}
Now we are in the position to state and prove the main result of this section.
\begin{theorem}
    \label{ThmCharOfMIsometricBWSByOperatorPolynomial}
    Let $ H $ be a Hilbert space and let $ S\in \mathbf{B}(\ell^{2}(\mathbb{Z},H)) $ be a bilateral weighted shift with invertible weights $ (S_{i})_{i\in \mathbb{Z}}\subset \mathbf{B}(H) $. For $ m\in \mathbb{Z}_{1} $ the following conditions are equivalent:
    \begin{enumerate}
        \item $ S $ is $ m $-isometric,
        \item there exists $ p\in \mathbf{B}(H)_{m-1}[z] $ such that $ p(0) = I $ and
        \begin{align}
            \label{FormMIsometricBWSPolynomialPositiveWeights}
            p(n) &= \lvert S_{n}\cdots S_{1}\rvert^{2}, \quad n\in \mathbb{Z}_{1},\\
            \label{FormMIsometricBWSPolynomialNegativeWeights}
            p(-n) &= \lvert S_{-n+1}^{\ast}\cdots S_{0}^{\ast}\rvert^{-2}, \quad n\in \mathbb{Z}_{1}.
        \end{align}
    \end{enumerate}
    Moreover, if (ii) holds, then
    \begin{equation}
        \label{FormNormOfMISometricBWS}
        \lVert S\rVert^{2} = \sup\left\{ \frac{\langle p(n+1)h,h\rangle}{\langle p(n)h,h\rangle}\colon n\in \mathbb{Z}, h\in H,h\not=0 \right\}.
    \end{equation}
\end{theorem}
\begin{proof}
    (i)$ \Longrightarrow $(ii). By Observation \ref{ObsOnOperatorsSk} and Theorem \ref{ThmCharOfMIsometricUWSByOperatorPolynomial}, for every $ k\in \mathbb{Z} $, $ k\le 0 $, there exists $ p_{k}\in \mathbf{B}(H)_{m-1}[z] $ such that $ p_{k}(0) = I $ and
    \begin{equation}
        p_{k}(n) = \lvert S_{k+n}\cdots S_{k+1}\rvert^{2}, \quad n\in \mathbb{Z}_{1}.
    \end{equation}
    Observe that
    \begin{align*}
        p_{k}(n-k) = \lvert S_{n}\cdots S_{k+1}\rvert^{2} &= S_{k+1}^{\ast}\cdots S_{0}^{\ast} \lvert S_{n}\cdots S_{1}\rvert^{2}S_{0}\cdots S_{k+1},\\
        &= S_{k+1}^{\ast}\cdots S_{0}^{\ast} p_{0}(n)S_{0}\cdots S_{k+1}\quad n\in \mathbb{Z}_{1}, k\in \mathbb{Z}, k< 0.
    \end{align*}
    Hence, we have
    \begin{equation*}
        p_{0}(n) = (S_{k+1}^{\ast}\cdots S_{0}^{\ast})^{-1}p_{k}(n-k)(S_{0}\cdots S_{k+1})^{-1}, \quad n\in \mathbb{Z}_{1}, k\in \mathbb{Z}, k< 0.
    \end{equation*}
    Since
    \begin{equation*}
        (S_{k+1}^{\ast}\cdots S_{0}^{\ast})^{-1}p_{k}(z-k)(S_{0}\cdots S_{k+1})^{-1} \in \mathbf{B}(H)_{m-1}[z], \quad k\in \mathbb{Z}, k<0,
    \end{equation*}
    it follows that
    \begin{equation*}
        (S_{k+1}^{\ast}\cdots S_{0}^{\ast})^{-1}p_{k}(z-k)(S_{0}\cdots S_{k+1})^{-1} = p_{0}(z), \quad z\in \mathbb{C}, k\in \mathbb{Z}, k<0.
    \end{equation*}
    In particular,
    \begin{align*}
        p_{0}(k) &= (S_{k+1}^{\ast}\cdots S_{0}^{\ast})^{-1}p_{k}(0)(S_{0}\cdots S_{k+1})^{-1}\\
        &=(S_{k+1}^{\ast}\cdots S_{0}^{\ast})^{-1}(S_{0}\cdots S_{k+1})^{-1}\\
        &=(S_{0}\cdots S_{k+1}S_{k+1}^{\ast}\cdots S_{0}^{\ast})^{-1}\\
        &=\lvert S_{k+1}^{\ast}\cdots S_{0}\rvert^{-2}, \quad k\in \mathbb{Z}, k<0.
    \end{align*}
    Therefore, $ p_{0} $ is the polynomial satisfying (ii).\\
    (ii)$ \Longrightarrow $(i). Arguing as in the proof of the converse implication, we can easily verify that for every $ k\in \mathbb{Z} $, $ k<0 $, the polynomial $ p_{k}\in \mathbf{B}(H)_{m-1}[z] $ given by
    \begin{equation}
        \label{ProofFormPolynomialForNegativeSk}
        p_{k}(z) = S_{k+1}^{\ast}\cdots S_{0}^{\ast} p(z+k)S_{0}\cdots S_{k+1}, \quad z\in \mathbb{C},
    \end{equation}
    satisfies Theorem \ref{ThmCharOfMIsometricUWSByOperatorPolynomial}(iii) for $ S_{k} $. Moreover, $ p $ satisfies Theorem \ref{ThmCharOfMIsometricUWSByOperatorPolynomial} for $ S_{0} $. Hence, $ S_{k} $ is $ m $-isometric for every $ k\in \mathbb{Z} $, $ k\le 0 $. Since $ S_{k}\subset S_{0} $ for every $ k\in \mathbb{Z}_{1} $, $ S_{k} $ is $ m $-isometric also for $ k\in \mathbb{Z}_{1} $. The application of Observation \ref{ObsOnOperatorsSk}(iii) gives (i).\\
    Now we prove the ''moreover'' part. Proceeding as in the proof of Theorem \ref{ThmCharOfMIsometricUWSByOperatorPolynomial} we obtain that
    \begin{equation*}
        \lVert S_{n+1}\rVert^{2} = \sup\left\{ \frac{\langle p(n+1)h,h\rangle}{\langle p(n)h,h\rangle}\colon h\in H, h\not=0 \right\}, \quad n\in \mathbb{N}.
    \end{equation*}
    Next, note that
    \begin{equation*}
        \lvert S_{-n+1}^{\ast}\cdots S_{0}^{\ast}\rvert^{-2} = S_{0}^{\ast -1}\cdots S_{-n+1}^{\ast -1}S_{-n+1}^{-1}\cdots S_{0}^{-1}, \quad n\in \mathbb{Z}_{1}
    \end{equation*}
    Hence, if $ n\in \mathbb{Z}_{2} $, then, by the fact that $ \mathcal{R}(S_{-n+1}^{-1}\cdots S_{0}^{-1}) = H $, we have
    \begin{align*}
        \lVert S_{-n+1}\rVert^{2} &= \sup\left\{ \frac{\lVert S_{-n+1}h\rVert^{2}}{\lVert h\rVert^{2}}\colon h\in H,h\not=0 \right\}\\
        &=\sup\left\{ \frac{\lVert S_{-n+1}S_{-n+1}^{-1}\cdots S_{0}^{-1}h\rVert^{2}}{\lVert S_{-n+1}^{-1}\cdots S_{0}^{-1}h\rVert^{2}}\colon h\in H,h\not=0 \right\}\\
        &=\sup\left\{ \frac{\langle S_{0}^{\ast -1}\cdots S_{-n+2}^{\ast -1}S_{-n+2}^{-1}\cdots S_{0}^{-1}h,h\rangle}{\langle S_{0}^{\ast -1}\cdots S_{-n+1}^{\ast -1}S_{-n+1}^{-1}\cdots S_{0}^{-1}h,h\rangle}\colon h\in H,h\not=0 \right\}\\
        &=\sup\left\{ \frac{\langle \lvert S_{-n+2}^{\ast}\cdots S_{0}^{\ast}\rvert^{-2} h,h\rangle}{\langle \lvert S_{-n+1}^{\ast}\cdots S_{0}^{\ast}\rvert^{-2}h,h\rangle}\colon h\in H,h\not=0 \right\}\\
        &=\sup\left\{ \frac{\langle p(-n+1)h,h\rangle}{\langle p(-n)h,h\rangle}\colon h\in H,h\not=0 \right\}.
    \end{align*}
    By the fact that $ p(0) = I $, arguing similarly as above we get that
    \begin{equation*}
        \lVert S_{0}\rVert^{2} = \left\{ \frac{\langle p(0)h,h\rangle}{\langle p(-1)h,h\rangle}\colon h\in H,h\not=0 \right\}.
    \end{equation*}
    Therefore, \eqref{FormNormOfMISometricBWS} follows from \eqref{FormNormOfBWS}.
\end{proof}
Below we state without the proof the counterpart of Corollary \ref{CorFormulaForWeightsUWSDoublyCommutingWeights} for bilateral weighted shifts.
\begin{corollary}
    Let $ H $ and $ S $ be as in Theorem \ref{ThmCharOfMIsometricBWSByOperatorPolynomial}.
    \begin{enumerate}
        \item If $ S $ is $ m $-isometric and the polynomial $ p\in \mathbf{B}(H)_{m-1}[z] $ is as in Theorem \ref{ThmCharOfMIsometricBWSByOperatorPolynomial}(ii), then $ \lvert S_{n}\rvert^{2} $ is unitarily equivalent to $ p(n-1)^{-1/2}p(n)p(n-1)^{-1/2} $ and $ \lvert S_{-n+1}^{\ast}\rvert^{2} $ is unitarily equivalent to $ p(-n+1)^{-1/2}p(-n)p(-n+1)^{-1/2} $ for every $ n\in \mathbb{Z}_{1} $.
        \item If the weights of $ S $ are doubly commuting, then $ m $-isometricity of $ S $ is equivalent to the existence of a polynomial $ p\in \mathbf{B}(H)_{m-1}[z] $ satisfying $ p(0) = I $ and
        \begin{align*}
            \lvert S_{n}\rvert^{2} &= p(n)p(n-1)^{-1} = p(n-1)^{-1}p(n), \quad n\in \mathbb{Z}_{1},\\
            \lvert S_{-n+1}^{\ast}\rvert^{2} &= p(-n)p(-n+1)^{-1} = p(-n+1)^{-1}p(-n), \quad n\in \mathbb{Z}_{1}.
        \end{align*}
    \end{enumerate}
\end{corollary}
As in the case of unilateral weighted shifts, we have the procedure of constructing of all $ m $-isometric bilateral weighted shifts with invertible and positive weights from a given polynomial $ p\in \mathbf{B}(H)_{m-1}[z] $ satisfying certain conditions.
\begin{theorem}
    \label{ThmConstructingMIsometricBWS}
    Let $ H $ be a Hilbert space. If $ p\in \mathbf{B}(H)_{m-1}[z] $ is a polynomial such that
    \begin{enumerate}
        \item $ p(0) = I $,
        \item $ p(n) $ is positive and invertible for every $ n\in \mathbb{Z} $,
        \item $ D(p):= \sup\left\{ \frac{\langle p(n+1)h,h\rangle}{\langle p(n)h,h\rangle}\colon n\in \mathbb{Z}, h\in H, h\not=0 \right\}<\infty $,
    \end{enumerate}
    then there exists an $ m $-isometric bilateral weighted shift with positive and invertible weights $ (S_{j})_{j\in \mathbb{Z}} $ satisfying \eqref{FormMIsometricBWSPolynomialPositiveWeights} and \eqref{FormMIsometricBWSPolynomialNegativeWeights}. Moreover, $ \lVert S\rVert = \sqrt{D(p)} $.
\end{theorem}
\begin{proof}
    Define the sequence $ (S_{j})_{j\in \mathbb{Z}}\subset \mathbf{B}(H) $ satisfying \eqref{FormMIsometricBWSPolynomialPositiveWeights} and \eqref{FormMIsometricBWSPolynomialNegativeWeights} by induction. Set $ S_{1} := p(1)^{1/2} $ and $ S_{0} := p(-1)^{-1/2} $. Clearly, $ S_{1} $ and $ S_{0} $ are positive and invertible and \eqref{FormMIsometricBWSPolynomialPositiveWeights}-\eqref{FormMIsometricBWSPolynomialNegativeWeights} hold for $ n=-1,0,1 $. Now assume that for some $ k\in \mathbb{Z}_{1} $, $ S_{-k+1},\ldots,S_{0},S_{1},\ldots,S_{k} $ are defined to be positive and invertible operators satisfying \eqref{FormMIsometricBWSPolynomialPositiveWeights}--\eqref{FormMIsometricBWSPolynomialNegativeWeights} for $ n=-k,\ldots,k $. Then
    \begin{align*}
        \widetilde{S}_{k+1} &:= S_{1}^{-1}\cdots S_{k}^{-1}p(k+1)S_{k}^{-1}\cdots S_{1}^{-1},\\
        \widetilde{S}_{-k} &:= S_{-k+1}^{-1}\cdots S_{0}^{-1}p(-k-1)S_{0}^{-1}\cdots S_{-k+1}^{-1}
    \end{align*}
    are positive and invertible and so are $ S_{k+1} := \widetilde{S}_{k+1}^{1/2} $ and $ S_{-k} := \widetilde{S}_{-k}^{-1/2} $. Proceeding as in the proof of Theorem \ref{ThmConstructingMIsometricUWS} we check that \eqref{FormMIsometricBWSPolynomialPositiveWeights} holds for $ n=k+1 $. To verify that \eqref{FormMIsometricBWSPolynomialNegativeWeights} holds for $ n=-k-1 $ observe that
    \begin{align*}
        \langle \lvert S_{-k}\cdots S_{0}\rvert^{-2}h,h\rangle &= \langle S_{0}^{-1}\cdots S_{-k+1}^{-1}\cdot \widetilde{S}_{-k}\cdot S_{-k+1}^{-1}\cdots S_{0}^{-1}h,h\rangle \\
        &= \langle p(-k-1)h,h\rangle.
    \end{align*}
    We will show that the sequence $ (S_{j})_{j\in \mathbb{Z}} $ defined above is uniformly bounded. Arguing as in the proof of Theorem \ref{ThmConstructingMIsometricUWS}, we have $ \lVert S_{n}\rVert^{2}\le D(p) $ for every $ n\in \mathbb{Z}_{1} $.
    In turn, by (iii),
    \begin{align*}
        \langle \lvert S_{-n+2}\cdots S_{0}\rvert^{-2}h,h\rangle &= \langle p(-n+1)h,h\rangle \\
        &\le D(p)\langle p(-n)h,h\rangle = \langle \lvert S_{-n+1}\cdots S_{0}\rvert^{-2}h,h\rangle, \quad n\in \mathbb{Z}_{2}.
    \end{align*}
    Substituting $ S_{0}\cdots S_{-n+1}h $ in place of $ h $ in the above inequality and using the fact that
    \begin{equation*}
        S_{-n+1}\cdots S_{0}\lvert S_{-n+2}\cdots S_{0}\rvert^{-2}S_{0}\cdots S_{-n+1} = S_{-n+2}^{2}
    \end{equation*}
    and
    \begin{equation*}
        S_{-n+1}\cdots S_{0}\lvert S_{-n+1}\cdots S_{0}\rvert^{-2}S_{0}\cdots S_{-n+1} = I,
    \end{equation*}
    we obtain that
    \begin{equation*}
        \lVert S_{-n+1}h\rVert^{2} = \langle S_{-n+1}^{2}h,h\rangle \le D(p)\langle h,h\rangle = D(p)\lVert h\rVert^{2}, \quad h\in H, n\in \mathbb{Z}_{2},
    \end{equation*}
    so $ \lVert S_{-n+1}\rVert^{2} \le D(p) $ for every $ n\in \mathbb{Z}_{2} $. Similar argument shows that $ \lVert S_{0}\rVert^{2}\le D(p) $. We deduce from Theorem \ref{ThmCharOfMIsometricBWSByOperatorPolynomial} that the weighted shift $ S\in \mathbf{B}(\ell^{2}(\mathbb{Z},H)) $ with weights $ (S_{j})_{j\in \mathbb{Z}} $ is $ m $-isometric and $ \lVert S\rVert = \sqrt{D(p)} $.
\end{proof}
\begin{remark}
    \label{RemSupremumBySpectralRadius}
    We can prove, in essentially the same way as in the proof of Observation \ref{ObsSupremumBySpectralRadius}, that if $ p\in \mathbf{B}_{m-1}(H) $ satisfies Theorem \ref{ThmConstructingMIsometricBWS}(i)--(ii), then
    \begin{equation*}
        D(p) = \sup\left\{ r(p(n+1)p(n)^{-1})\colon n\in \mathbb{Z} \right\}.
    \end{equation*}
\end{remark}
As regards commutativity of weights of $ m $-isometric bilateral weighted shift, we have the following result (cf. Theorem \ref{ThmWhenWeightsCommuteUWS}).
\begin{theorem}
    \label{ThmWhenWeightsCommuteBWS}
    Let $ H $ be a Hilbert space and let $ S\in \mathbf{B}(\ell^{2}(\mathbb{Z},H)) $ be an $ m $-isometric bilateral weighted shift with positive and invertible weights $ (S_{j})_{j\in \mathbb{Z}}\subset \mathbf{B}(H) $. The following conditions are equivalent:
    \begin{enumerate}
        \item all weights of $ S $ 
        commute,
        \item $ S_{1},\ldots,S_{m-1} $ commute,
        \item if $ p(z) = A_{m-1}z^{m-1}+\ldots+A_{1}z+I \in \mathbf{B}(H)_{m-1}[z] $ satisfies Theorem \ref{ThmWhenWeightsCommuteBWS}(ii), then $ A_{1},\ldots,A_{m-1} $ commute.
    \end{enumerate}
\end{theorem}
\begin{proof}
    The implications (i)$ \Longrightarrow $(ii) and (ii)$ \Longrightarrow $(iii) may be proved in exactly the same way as in Theorem \ref{ThmWhenWeightsCommuteUWS}.\\
    (iii)$ \Longrightarrow $(i). By Observation \ref{ObsOnOperatorsSk}, $ S_{k} $ is $ m $-isometric for every $ k\in \mathbb{Z} $, $ k\le 0 $. It is sufficient to show that for every $ k<0 $ the weights of $ S_{k} $ commute. Proceeding as in the proof of Theorem \ref{ThmCharOfMIsometricBWSByOperatorPolynomial}, we get that $ p $ is the polynomial satisfying Theorem \ref{ThmCharOfMIsometricUWSByOperatorPolynomial}(iii) for $ S_{0} $ and the polynomials $ p_{k} $ ($ k<0 $) satisfying Theorem \ref{ThmCharOfMIsometricUWSByOperatorPolynomial}(iii) for $ S_{k} $ are given by \eqref{ProofFormPolynomialForNegativeSk}. By Theorem \ref{ThmWhenWeightsCommuteUWS}, it is enough to prove that the coefficients of $ p_{k} $ commute. From \eqref{ProofFormPolynomialForNegativeSk} it follows that
    \begin{equation*}
        p_{k}(z) = \sum_{j=1}^{m-1}B_{k}^{\ast}A_{j}B_{k}(z+k)^{j}+B_{k}^{\ast}B_{k}, \quad k\in \mathbb{Z}, k<0,
    \end{equation*}
    where $ B_{k} = S_{0}\cdots S_{k+1} $ ($ k<0 $). By Lemma \ref{LemOperatorPolynomialFromCharOfMIsometricUWS}, for every $ j=1,\ldots,m-1 $, $ A_{j} $ is a linear combination of operators $ I,\lvert S_{1}\rvert^{2},\cdots,\lvert S_{m-1}\cdots S_{1}\rvert^{2} $ with real coefficients. Hence, it suffice to show that for every $ j,\ell\in \left\{ 1,\ldots m-1\right\} $,
    \begin{equation}
        \label{ProofFormCoeffsCommute1}
        B_{k}^{\ast}\lvert S_{j}\cdots S_{1}\rvert^{2}B_{k}B_{k}^{\ast}\lvert S_{\ell}\cdots S_{1}\rvert^{2}B_{k} = B_{k}^{\ast}\lvert S_{\ell}\cdots S_{1}\rvert^{2}B_{k}B_{k}^{\ast}\lvert S_{j}\cdots S_{1}\rvert^{2}B_{k}
    \end{equation}
    and
    \begin{equation}
        \label{ProofFormCoeffsCommute2}
        B_{k}^{\ast}\lvert S_{j}\cdots S_{1}\rvert^{2}B_{k}B_{k}^{\ast}B_{k} = B_{k}^{\ast}B_{k}B_{k}^{\ast}\lvert S_{j}\cdots S_{1}\rvert^{2}B_{k}.
    \end{equation}
    Since $ B_{k} $ is invertible, \eqref{ProofFormCoeffsCommute1} and \eqref{ProofFormCoeffsCommute2} are equivalent to
    \begin{equation*}
        \lvert S_{j}\cdots S_{1}\rvert^{2}B_{k}B_{k}^{\ast}\lvert S_{\ell}\cdots S_{1}\rvert^{2} = \lvert S_{\ell}\cdots S_{1}\rvert^{2}B_{k}B_{k}^{\ast}\lvert S_{j}\cdots S_{1}\rvert^{2}
    \end{equation*}
    and
    \begin{equation*}
        \lvert S_{j}\cdots S_{1}\rvert^{2}B_{k}B_{k}^{\ast} = B_{k}B_{k}^{\ast}\lvert S_{j}\cdots S_{1}\rvert^{2},
    \end{equation*}
    respectively. But, by the equality $ B_{k}B_{k}^{\ast} = \lvert S_{k+1}\cdots S_{0}\rvert^{2} $, we deduce from (iii) that
    \begin{align*}
        \lvert S_{j}\cdots S_{1}\rvert^{2}B_{k}B_{k}^{\ast}\lvert S_{\ell}\cdots S_{1}\rvert^{2} &= p(j)p(k)p(\ell) \\
        &= p(\ell)p(k)p(j) = \lvert S_{\ell}\cdots S_{1}\rvert^{2}B_{k}B_{k}^{\ast}\lvert S_{j}\cdots S_{1}\rvert^{2}
    \end{align*}
    and
    \begin{align*}
        \lvert S_{j}\cdots S_{1}\rvert^{2}B_{k}B_{k}^{\ast} = p(j)p(k) = p(k)p(j) = B_{k}B_{k}^{\ast}\lvert S_{j}\cdots S_{1}\rvert^{2}
    \end{align*}
    for every $ j,\ell\in \left\{ 1,\ldots,m-1 \right\} $. Thus, \eqref{ProofFormCoeffsCommute1} and \eqref{ProofFormCoeffsCommute2} hold, so the coefficients of $ p_{k} $ commute. By Theorem \ref{ThmWhenWeightsCommuteUWS}, the weights of $ S_{k} $ commute for every $ k\in \mathbb{Z} $, $ k<0 $.
\end{proof}

In \cite[Problem 4.9]{bermudezLocalSpectral2024} Berm\'{u}dez et.al. posed a question whether the adjoint of $ m $-isometric invertible operator is also $ m $-isometric. In the following result we characterize bilateral weighted shifts with operator weights, which possess this property.
\begin{theorem}
    \label{ThmAdjointIsMIsometricBWS}
    Let $ H $ be a Hilbert space and let $ S $ be an $ m $-isometric bilateral weighted shift with invertible weights $ (S_{j})_{j\in \mathbb{Z}} $. The following conditions are equivalent:
    \begin{enumerate}
        \item $ S^{\ast} $ is $ m $-isometric,
        \item the polynomial $ p\in \mathbf{B}(H)_{m-1}[z] $ as in Theorem \ref{ThmCharOfMIsometricBWSByOperatorPolynomial}(ii) is invertible in $ \mathbf{B}(H)_{m-1}[z] $, that is, there exists a polynomial $ \widetilde{p}\in \mathbf{B}(H)_{m-1}[z] $ such that
        \begin{equation}
            \label{FormMIsometricAdjointByInvertiblePolynomial}
            p(z) \cdot \widetilde{p}(z) = \widetilde{p}(z)\cdot p(z) =I, \quad z\in \mathbb{C}.
        \end{equation}
    \end{enumerate}
\end{theorem}
\begin{proof}
    By Lemma \ref{LemOnAdjointOfBWS}, $ S^{\ast} $ is unitarily equivalent to the bilateral weighted shift $ \widehat{S} \in \mathbf{B}(\ell^{2}(\mathbb{Z},H)) $ with weights $ (S_{-j+1}^{\ast})_{j\in \mathbb{Z}} $; $ S^{\ast} $ is $ m $-isometric if and only if so is $ \widehat{S} $. By Theorem \ref{ThmCharOfMIsometricBWSByOperatorPolynomial}, $ \widehat{S} $ is $ m $-isometric if and only if there exists a polynomial $ \widehat{p}\in \mathbf{B}(H)_{m-1}[z] $ such that
    \begin{align}
        \label{ProofFormPolynomialOfAdjointZero}
        \widehat{p}(0) &= I,\\
        \label{ProofFormPolynomialOfAdjointPositive}
        \widehat{p}(n) &= \lvert S_{-n+1}^{\ast}\cdots S_{0}^{\ast}\rvert^{2}, \quad n\in \mathbb{Z}_{1},\\
        \label{ProofFormPolynomialOfAdjointNegative}
        \widehat{p}(-n) &= \lvert S_{n}\cdots S_{1}\rvert^{-2}, \quad n\in \mathbb{Z}_{1}.
    \end{align}
    (i)$ \Longrightarrow $(ii). Let $ \widehat{p}\in \mathbf{B}(H)_{m-1}[z] $ satisfy \eqref{ProofFormPolynomialOfAdjointZero}-\eqref{ProofFormPolynomialOfAdjointNegative} and define $ \widetilde{p}(z) = \widehat{p}(-z) $, $ z\in \mathbb{C} $. Then $ \widetilde{p}\in \mathbf{B}(H)_{m-1}[z] $ and
    \begin{align*}
        p(0)\widetilde{p}(0) &= p(0)\widehat{p}(0) = I\cdot I = I,\\
        p(n)\widetilde{p}(n) &= p(n)\widehat{p}(-n) = \lvert S_{n}\cdots S_{1}\rvert^{2}\lvert S_{n}\cdots S_{1}\rvert^{-2} = I, \quad n\in \mathbb{Z}_{1},\\
        p(-n)\widetilde{p}(-n) &= p(-n)\widehat{p}(n) = \lvert S_{-n+1}^{\ast}\cdots S_{0}^{\ast}\rvert^{-2}\lvert S_{-n+1}^{\ast}\cdots S_{0}^{\ast}\rvert^{2} = I, \quad n\in \mathbb{Z}_{1}.
    \end{align*}
    Since $ p\cdot \widetilde{p}\in \mathbf{B}_{2m-2}[z] $ and $ p(n)\widetilde{p}(n) = I $ for every $ n\in \mathbb{Z} $, it follows that $ p\cdot \widetilde{p} = I $. Similarly, $ \widetilde{p}\cdot p = I $.\\
    (ii)$ \Longrightarrow $(i). Proceeding as in the proof of the converse implication, we get that the polynomial $ \widehat{p}(z) = \widetilde{p}(-z) $, $ z\in \mathbb{C} $, satisfies \eqref{ProofFormPolynomialOfAdjointZero}-\eqref{ProofFormPolynomialOfAdjointNegative}.
\end{proof}
As a corollary, we obtain that the only $ m $-isometric classical bilateral weighted shifts (with scalar weights), for which the adjoint is also $ m $-isometric, are unitary. This answers the question in \cite[Problem 4.9]{bermudezLocalSpectral2024} in negative.
\begin{corollary}
    \label{CorClassicalBWSWhenAdjointIsMIsometric}
    Let $ S\in \mathbf{B}(\ell^{2}(\mathbb{Z},\mathbb{C})) $ be an $ m $-isometric bilateral weighted shift with weights $ (\lambda_{j})_{j\in \mathbb{Z}}\subset \mathbb{C}\setminus\{0\} $. The following conditions are equivalent:
    \begin{enumerate}
        \item $ S^{\ast} $ is $ m $-isometric,
        \item $ S $ is unitary.
    \end{enumerate}
\end{corollary}
\begin{proof}
    It suffices to prove that (i) implies (ii). We have $ \mathbf{B}(\mathbb{C}) = \mathbb{C} $ (up to an isomorphism). Because the only invertible elements of $ \mathbb{C}[z] $ are constant polynomials, it follows that in Theorem \ref{ThmAdjointIsMIsometricBWS}(ii), it has to be $ p = 1 $, so $ S $ is isometric. Since, by \eqref{FormWhenBWSIsInvertible}, $ S $ is invertible, we get that $ S $ is unitary.
\end{proof}
In general, the conclusion of Corollary \ref{CorClassicalBWSWhenAdjointIsMIsometric} is not true if $ \dim H \ge 2 $ as is shown in the following example.
\begin{example}
    Define $ p\in \mathbf{B}(\mathbb{C}^{2})_{2}[z] $ by the formula:
    \begin{equation*}
        p(z) = \begin{bmatrix}
            1 & z\\
            z & z^{2}+1
        \end{bmatrix}, \quad z\in \mathbb{C}.
    \end{equation*}
    It is a matter of routine to verify that $ p(n) $ is positive and invertible for every $ n\in \mathbb{Z} $.  Moreover, for every $ n\in \mathbb{Z} $ the characteristic polynomial of $ p(n+1)p(n)^{-1} $ is equal to $ h(z) = z^{2}-3z+1 $. Hence,
    \begin{equation*}
        r(p(n+1)p(n)^{-1}) = r(p(1)p(0)^{-1}), \quad n\in \mathbb{Z}.
    \end{equation*}
    Therefore, $ \sup\left\{ r(p(n+1)p(n)^{-1})\colon n\in \mathbb{Z} \right\} < \infty $. By Remark \ref{RemSupremumBySpectralRadius} and Theorem \ref{ThmConstructingMIsometricBWS} there exists a $ 3 $-isometric bilateral weighted shift $ S\in \mathbf{B}(\ell^{2}(\mathbb{Z},\mathbb{C}^{2})) $ with positive and invertible weights $ (S_{j})_{j\in \mathbb{Z}} $ satisfying \eqref{FormMIsometricBWSPolynomialPositiveWeights} and \eqref{FormMIsometricBWSPolynomialNegativeWeights}. Set
    \begin{equation*}
        q(z) := \begin{bmatrix}
            z^{2}+1 & -z\\
            -z & 1
        \end{bmatrix}, \quad z\in \mathbb{C}.
    \end{equation*}
    Then $ q\in \mathbf{B}(\mathbb{C}^{2})_{2}[z] $ and $ p\cdot q = q\cdot p = I $. Hence, by Theorem \ref{ThmAdjointIsMIsometricBWS}, $ S^{\ast} $ is also 3-isometric.
\end{example}

	\bibliographystyle{plain}

\end{document}